%% file: BLtensorLek.tex
\documentclass[11pt]{amsart}
\usepackage{latexsym,amssymb,amsmath}
\usepackage{multirow}
\input BLtensordefs.tex

\numberwithin{equation}{section}
\numberwithin{theorem}{section}

\def\set#1{\left\{#1\right\}}
\def\fromto#1#2{#1, \dotsc, #2}
\def\setfromto#1#2{\set{\fromto{#1}{#2}}}

\def\ccI{\mathcal I}
\def\ccO{\mathcal O}

\newcommand{\JaBu}{Jaros\L{}aw Buczy\'n{}ski}
\newcommand{\shortJaBu}{J.~Buczy\'n{}ski}

\DeclareMathOperator{\codim}{codim}
\def\ur{\underline{\mathbf{R}}}

\renewcommand{\theenumi}{(\roman{enumi})}
\renewcommand{\labelenumi}{\theenumi}

\def\segduals{\mathcal S  e  g^{\vee}_*}
\def\bv{\bold v}\def\bw{\bold w}

\title[Ranks of tensors]{Ranks of tensors and  a generalization of secant varieties}
\author[\shortJaBu{} \& J.M.~Landsberg]{\JaBu \and  J.M. Landsberg}

\begin{document}

\begin{abstract} We introduce subspace rank as a tool for studying ranks of tensors and $X$-rank more generally.
We derive a new upper bound for the rank of a tensor and determine the ranks of partially symmetric
tensors in $\BC^2\ot \BC^{\bbb}\ot \BC^{\bbb}$. We review the literature from a geometric perspective.
\end{abstract}
 \thanks{\shortJaBu{} supported by Marie Curie Outgoing Fellowship ``Contact Manifolds'',
         Landsberg supported by NSF grant   DMS-1006353}
\email{jabu@mimuw.edu.pl, jml@math.tamu.edu}

\address{Institut Fourier \\ Universit\'e Grenoble I \\
100 rue des Maths, BP 74\\
38402 St Martin d'H\`eres cedex, France\\
and
Institute of Mathematics of the
Polish Academy of Sciences\\
ul. \'Sniadeckich 8\\
P.O. Box 21\\
00-956 Warszawa, Poland}

\address{Department of Mathematics\\
Texas A\&M University\\
Mailstop 3368\\
College Station, TX 77843-3368, USA}

\maketitle

\section{Introduction}

A central problem in many areas (signal processing, algebraic statistics, complexity theory etc.,
  see  e.g., \cite{BCS,MR2447451,MR2205865,MR2383305})  is to understand the ranks and border ranks of  tensors,
and the analogous notions for symmetric and partially symmetric tensors.  There have been many recent
contributions to this study using methods from algebraic geometry, a few examples are
\cite{MR2452824,BGI,jabu_ginensky_landsberg_Eisenbuds_conjecture,MR2392585,MR2383305,LTrank}.
In particular, the notions of {\it $X$-rank} and {\it $X$-border rank}, defined below, allow  one to treat questions
regarding tensors, symmetric tensors, and partially symmetric tensors uniformly.
In this paper we prove a new upper bound for the maximum rank of a tensor as a Corollary of
a new upper bound for $X$-rank in general, see Proposition \ref{prop:bound_on_superspace_rank_LT} and Corollary \ref{tenrkbnd}. For example, we show the maximum rank of a tensor
in $\BC^n\ot \BC^n\ot \BC^n$ is bounded by $n^2-n-1$. (The previous known upper bound had
been $n^2$.)  We also show that the partially symmetric version
of Comon's conjecture holds in $\BC^2\ot \BC^{\bbb}\ot \BC^{\bbb}$. For tensors in $\BC^2\ot \BC^{\bbb}\ot \BC^{\ccc}$,
with $\bbb\leq 3$, where there exist normal forms, for each normal form we give a geometric interpretation
of the point and determine its rank and border rank. A substantial part of this paper is   essentially
expository. The notion of {\it subspace border rank}   dates back to Terracini. We define a corresponding
notion of subspace rank, review the literature and establish basic properties. The key to studying tensors
in $\BC^2\ot \BC^{\bbb}\ot \BC^{\ccc}$ is Kronecker's normal form for pencils of matrices, which we review as a prelude
to a proof of the Grigoriev-Ja'Ja'-Teichert theorem on ranks of pencils, which we generalize to symmetric tensors, Theorem \ref{sympencilthm}.

 \subsection{Definitions}
Let $A_1\hd A_k,A,B,C,V,W$ be complex
vector spaces of dimensions, respectively,
 $\aaa_1\hd \aaa_k$, $\aaa,\bbb,\ccc,\bv,\bw$.  Let  $Seg(\BP A_1\ctimes \BP A_k)
\subset \BP (A_1\otc A_k)$ denote the Segre variety of rank one
tensors.  Since the property of   being rank one is invariant under
scalar multiplication,   it is natural to  quotient out by   rescalings  and work in the corresponding projective
space.
Given a tensor
$p\in A_1\otc A_k$, define the {\it (tensor) rank}  (resp.~{\it (tensor) border rank})
$\mathbf{R}(p)= \mathbf{R}_{Seg(\BP A_1\ctimes \BP A_k)}([p])$ (resp.
${\ur}(p)= {\ur}_{Seg(\BP A_1\ctimes \BP A_k)}([p])$)
of $p$ to be the smallest $r$ such that there exist
\[
  [a^1_1\otc a^1_k]\hd [a^r_1\otc a^r_k]\in Seg(\BP A_1\ctimes \BP A_k)
\]
 with $a^i_j\in A_j$ such that
$p=a^1_1\otc a^1_k+\cdots + a^r_1\otc a^r_k$  (resp.~that there exist
curves $a^i_j(t)\subset A_j$ such that
$[p]=\tlim_{t\ra 0} [a^1_1(t)\otc a^1_k(t)+\cdots + a^r_1(t)\otc a^r_k(t)]$).
Here we write limits in projective space $\BC\pp N$ as opposed to limits
in affine space $\BC^{N+1}$. Although essentially equivalent, taking limits in
projective space has the advantage of eliminating a scalar to worry about.
These definitions agree  with the definitions
in the tensor literature.

More generally, for a projective variety $X\subset \BP V$ not contained in
a hyperplane, the
$X$-rank of $[p]\in \BP V$, $\mathbf{R}_X([p])$,  is defined to be the smallest $r$ such that
there exist $x_1\hd x_r\in X$ such that 
$[p]$ is in the linear span 
of  $x_1\hd x_r$
and the $X$-border rank $\ur_X([p])$ is defined as the smallest integer, such that $p$
  is a limit of points of $X$-rank $r$. When $X=Seg(\BP A_1\ctimes \BP A_k)$ one recovers
the definitions of (tensor) rank and border rank.

\medskip

Let    $S^dV$ (resp. $\La d V$) denote the space of  symmetric (resp. skew-symmetric) tensors in $V^{\ot d}$.
Let $v_d(\BP V)\subset \BP (S^d V)$
denote the Veronese variety of $d$-th powers of linear forms.  Let
$G(k, V)$ denote the Grassmannian of $k$-planes through the origin
in $V$.
The Grassmannian may be viewed as a projective variety in $\BP(\La k V)$,
via the {\it Pl\"ucker embedding}, where, for a $k$-plane
$E$, take a basis $e_1\hd e_k$ and map  $E$ to $[e_1\ww\cdots \ww e_k]$.
 For a subset $Z\subset \BP V$,
  let $\langle Z\rangle \subseteq V$ denote its linear span and $\hat Z\subset V$ the associated cone in $V$.

The following definition can be traced back to Terracini \cite{Terracini2}, who asked:
 given $d,k,n$, what is the smallest  $r$ such that a general collection of $k$
  homogeneous polynomials of degree $d$ in $n$ variables
 may be expressed as linear combinations
of $d$-th powers of the same linear forms
$l_1\hd l_r$?
Grassmann secant varieties were the  subject of the 2001
EAGER Summer School PRAGMATIC, see \cite{MR2009896}.

 \begin{definition}\label{subspacerk}
For a variety $X\subset \BP V$ not contained in a hyperplane,  and $E\in G(k,V)$, define
  $\mathbf{R}_X(E)$, the {\it $X$-rank of $E$}  to be the smallest
$r$, such that there exists $x_1\hd x_r\in X$ and
$E\subset \langle x_1\hd x_r\rangle$. Define
$\s_{r,k}^0(X)\subset G(k,V)$ to be the set of $k$-planes of
$X$-rank at most $r$, and let $\s_{r,k}(X)\subset G(k,W)$ denote
its Zariski closure, called the  \emph{Grassmann secant variety}. When $k=1$ we write
$\s_{r,1}(X)=\s_r(X)\subset \BP V$ for the $r$-th secant variety of $X$.
Our notation is such that $\s_1(X)=X$.
We remark that $\s_{r,k}(X)$ is denoted $G_{k-1,r-1}(X)$ in much of the literature.
\end{definition}

Terracini's question is: what is  the smallest value of $r$ such that $\s_{r,k}(v_d(\BP V))=G(k,V)$?

The notion of subspace border rank re-appeared in the complexity literature.
Strassen \cite{Strassen505} had the idea to reduce the study of rank and border rank of  tensors in
$A\ot B\ot C$ to the study of linear subspaces of spaces of
endomorphisms, by considering $p\in A\ot B\ot C$ as a linear map $A^*\ra B\ot C$ and studying the image.
Strassen's famous equations for
the set of tensors of border rank
at most three in $\BC^3\ot\BC^{\bbb}\ot \BC^{\bbb}$,  are,
after fixing an identification $C\simeq B^*$ so the image may be thought of as a space
of endomorphisms,
   exactly the expression
that  the endomorphisms commute (see \cite{BCS,LMsecb} for discussions).

\medskip

\begin{remark}
    Grassmann secant varieties have recently appeared in
  \cite{chiantini_coppens_Grassmannians_of_secant_varieties},
    \cite{ciliberto_cools_Grassmann_secant_extremal_varieties},
    \cite{cools_Singular_locus_of_Grassmann_secant},  \cite{chipalkatti_Waring_locus},
        \cite{chipalkatti_geramita_Artin_level_algebras}, and
        \cite{carlini_chipalkatti_Waring_for_several_forms}. These articles are primarily interested
in studying the dimension of $\s_{r,k}(X)$, especially in the case $X=v_d(\BP W)$ is a Veronese variety.
While the $X$-border rank of linear spaces has
been well studied, to our knowledge, our paper initiates the study of $X$-rank of linear spaces.
The $X$-(border) rank of linear spaces is related to $Seg(\PP^{\aaa-1} \times X)$-(border) rank of points,
  see Theorem~\ref{thm:can_calculate_rank_as_superspace_rank} below.
The special cases where $X$ is a Veronese variety arise in applications (see e.g., \cite{MR2473189} \cite{MR2517853}).
\end{remark}

\subsection{Overview}
In \S\ref{stdrevistsect} we revisit two standard results on ranks of tensors to generalize
and strengthen them
(Proposition~\ref{exprthm} and Theorem~\ref{thm:can_calculate_rank_as_superspace_rank}).
In \S\ref{genlfactssect} we establish basic properties about the varieties $\s_{r,k}(X)$ and
$X$-ranks of linear spaces. We briefly mention in \S\ref{coincidesect} a few cases where ranks
and border ranks of tensors agree.
In
\S\ref{jajasect} we give an exposition of  work of Grigoriev,  Ja'Ja'
and Teichert \cite{MR0453768,MR519843,MR521052,Teichert} where the
 ranks of tensors in
$\BC^2\ot \BC^{\bbb}\ot \BC^{\ccc}$ are completely determined.
When $\bbb\leq 3$ we recall 
in \S\ref{section_orbits}
the  normal forms for
such tensors from \cite{MR1830474}   and give geometric interpretations
for the points of a given normal form.  In \S\ref{parsymsect} we apply the results
to the case of partially symmetric tensors, proving an analog of the
Grigoriev-Ja'Ja'-Teichert theorem. 
In \S\ref{agproofs}
we conclude with proofs of results stated
in \S\ref{genlfactssect} that require additional  terminology from algebraic geometry.

\subsection{Acknowledgments}
We thank J. Weyman
for   pointing out
  there were orbits missing
in the list in \S\ref{section_orbits}
 in an earlier version of this paper.
We also thank  anonymous readers of an earlier version of this
paper for numerous corrections regarding the history of questions
and for useful suggestions regarding the exposition, in particular correcting an
error in Lemma 8.1.
This paper grew out of questions raised at the
2008 AIM workshop {\it Geometry and representation theory of tensors
for computer science, statistics and other areas}, and the authors
thank AIM and the conference participants for inspiration.

\section{Generalizations of standard results on ranks of Segre products}\label{stdrevistsect}

Let  $Y\subset \BP W$ be a variety and let
$X=Seg(\BP A\times Y)\subset \BP (A\ot W)$ be the
Segre product of $Y$ with the projective space $\BP A$.

\begin{proposition}\label{exprthm} Let $A'\subset A$ be
a linear subspace and let  $p\in \BP (A'\ot W)$. Then any expression
$p=[v_1+\cdots +v_s]$
such that some $[v_j]\not\in
X\cap \BP (A'\ot W)$ has $s>\mathbf{R}_X(p)$.

For any expression $p=\tlim_{t\ra 0}  [v_1(t)+\cdots +v_s(t)]$ with $[v_j(t)] \in X$   there exist  $w_1(t),\dotsc,w_s(t)$,
        such that $p=\tlim_{t\ra 0}  [w_1(t)+\cdots +w_s(t)]$ with $[w_j(t)] \in X \cap \BP(A' \otimes W)$.

In particular $\mathbf{R}_X|_{\BP(A' \otimes W)} = \mathbf{R}_{X \cap \BP(A' \otimes W)}$
and $\ur_X|_{\BP(A' \otimes W)} = \ur_{X \cap \BP(A' \otimes W)}$.
\end{proposition}

Proposition~\ref{exprthm} recovers and strengthens (the \lq\lq moreover\rq\rq\  statement below) the following standard fact
(e.g., \cite[Prop. 14.35]{BCS}, \cite[Prop. 3.1]{MR2447444}):

\begin{corollary}\label{zerowash} Let $n>2$.
Let $T\in A_1\otc A_n $ have rank $r$. Say
$T\in A_1'\otc A_n'$, where $A'_j\subseteq A_j$,   with at least one inclusion proper.
Then any expression   $T=\sum_{i=1}^{\rho}u^1_i\otc u^n_i$  with
some $u^s_j\not\in A_s'$ has $\rho>r$. In particular
$\mathbf{R}_{Seg(\BP A_1\ctimes \BP A_n)}(T)=\mathbf{R}_{Seg(\BP A_1'\ctimes \BP A_n')}(T)$.
Moreover
$\ur_{Seg(\BP A_1\ctimes \BP A_n)}(T)=\ur_{Seg(\BP A_1'\ctimes \BP A_n')}(T)$.
\end{corollary}

\begin{remark}
  In the language of \cite[\S4]{jabu_ginensky_landsberg_Eisenbuds_conjecture}
  Proposition~\ref{exprthm} says (in particular) that $(X, A'\otimes W)$ is $rpp$ and $brpp$,
  i.e., a (border) rank preserving pair.
\end{remark}

\begin{proof}[Proof of Proposition  \ref{exprthm}.]
 Choose a complement $A''\subset A$  to $A'$ so that  $A=A'\op A''$.

To prove the assertion regarding rank, write $p=v_1+\cdots + v_s$,
where $v_j=a_j\ot y_j$ with $y_j\in \hat Y$.
Write $a_j=b_j+c_j$ with $b_j\in A'$ and $c_j\in A''$. Since
$p\in A'\ot W$, we have $\sum c_j\ot y_j=0$.
Let $\{ e_m\}$ be a basis of
$A''$. We have $c_j=\xi^m_j e_m$ so
$\sum_j \xi^m_j y_j=0$ for all $m$. Say e.g.,  $\xi^1_s\neq 0$, then
we can write $y_s$ as a linear combination of $y_1\hd y_{s-1}$ and
obtain an expression of rank $s-1$.

To prove the border rank assertion,
for each $v_j(t)=a_j(t)\ot y_j(t)$ write $a_j(t)=b_j(t)+c_j(t)$ with
$b_j(t) \subset A'$ and $c_j(t) \subset A''$.
Also write $p(t) := \sum v_j(t)$,
so that $[p]=\tlim_{t\ra 0} [p(t)]$.
Since $[p] \in \BP (A'\otimes W)$,
  \[
    [p]=\tlim_{t\ra 0} [ \sum b_j(t)\ot y_j(t)],
  \]
proving the claim.
\end{proof}

 The following bound on rank (in the three factor case) appears in
\cite[p 231, IV.3]{MR0398167}, as well as
 \cite[Prop. 14.45]{BCS} as an inequality, where it is said to be classical,
  and as an equality  in
 \cite[Thm 2.4]{Friedland3tensor}. It is also used e.g., in \cite{MR521052}:

\begin{proposition}\label{rankbndthm}  Let $\phi\in A_1\otc A_n$.
Then $\mathbf{R}(\phi)$ equals  the number of points of $\hat Seg(\BP A_2\ctimes \BP A_n)$
needed to span a space containing $\phi(A_1^*)\subset A_2\otc A_n$ (and similarly for the permuted statements).
Here we have interpreted $\phi$ as a linear map:
  $\phi : (A_1)^* \to A_2\otc A_n$.
\end{proposition}

Proposition \ref{rankbndthm} has the following generalization, whose border rank version appeared
 in \cite{MR2009896}, and whose  proof dates back to Terracini \cite{Terracini2}:

\begin{theorem}\label{thm:can_calculate_rank_as_superspace_rank}
  Let $Y \subset \BP W$ and let  $X:= Seg(\BP A \times Y)$.
Given $ p \in A \otimes W$, then
    $\mathbf{R}_X ([p]) = \mathbf{R}_Y(p(A^*))$ and $\ur_X([p])=\ur_Y(p(A^*))$,  where on the right hand
sides of the equations we have interpreted $p$ as a linear map:
  $p : A^* \to W$.
\end{theorem}

\begin{proof}
We  prove the border rank statement,
  the rank statement is the special case where each curve is constant.

To see  $\ur_X ([p]) \ge \ur_Y(p(A^*))$, assume $\ur_X([p])=r$ and write
  \[
     p(t)= a_1(t) \otimes y_1(t) + \dotsb + a_r(t) \otimes y_r(t)
  \]
  where $a_i(t) \in A$ and $y_i(t) \in \hat Y$ and
  \[
     [p] = \tlim_{t\ra 0} [p(t)].
  \]
 If $A' \subset A$ is such that $[p] \in \BP(A' \otimes W)$,
  then by Proposition~\ref{exprthm}   we may assume
  \[
    p(t) \in A' \otimes W \text{ and } a_i(t) \in A'.
  \]
Replacing   $A$ by a smaller vector space if necessary, we may assume
$p(t): A^*\ra W$ is injective for all values of $t$ sufficiently close to $0$.
Thus the image of $p(t): A^* \ra W$ determines a curve in the Grassmannian $G(\dim A, W)$,
Since $[p]=\tlim_{t\ra 0}[p(t)]$ as a linear map defined up to scale,
$p(A^*)\subset \lim_{t\ra 0}\langle y_1(t)\hd y_r(t)\rangle$,
where the limit is taken in $G(r, W)$.
Thus $\mathbf{R}_Y(p(t)(A^*)) \le r$ and $\ur_Y(p(A^*)) \le r = \ur_X ([p])$.

To see $\ur_X ([p]) \le \ur_Y(p(A^*))$, assume $\ur_Y(p(A^*))=r$ and that there exist curves
$y_j(t)\subset \hat Y$, such that $p(A^*)\subset \tlim_{t\ra 0}\langle y_1(t)\hd y_r(t)\rangle$,
where again we may assume the dimension of the span of the $y_s(t)$ is constant for $t\neq 0$ and
we are taking the limit in the appropriate Grassmannian. Let $a^1\hd a^{\dim A}$ be a basis of
$A^*$. Write $p(a^j)=\tlim_{t\ra 0}\sum_s c^{js}(t)y_s(t) $ for some functions $c^{js}(t)$,
so $p=\sum_j a_j\ot [\tlim_{t\ra 0}\sum_s c^{js}(t)y_s(t)]$.
Consider the curve
\[
   p(t) =\sum_{j,s} a_j\ot c^{j s}(t) y_s(t) =\sum_{ s=1}^r \left[ \sum_j c^{j s}(t) a_j \right] \ot y_s(t) .
\]
Thus $\mathbf{R}_X([p(t)]) \le r$ for $t\neq 0$ and since $\tlim_{t\ra 0}[p(t)]=[p]$, the claim $\ur_X([p])\leq r$ follows.
\end{proof}

\section{General facts about rank and border rank of linear spaces}\label{genlfactssect}

The following facts are immediate consequences of the definitions of $\mathbf{R}_X(E),\ur_X(E)$:

\begin{proposition}\label{elemrank}
 Let $X\subset \BP V=\pp N$ be a variety of dimension $n$ not
contained in a hyperplane, let   $E\in G(k,V)$. Then
\begin{enumerate}
\item $\mathbf{R}_X(E)\ge \ur_X(E) \ge k$ and $\sigma_{r_1,k}(X)\subset \sigma_{r_2,k}(X)$ whenever $r_1 \le r_2$.
\item If $X$ is irreducible and $E\in \s_{r,k}(X)$ is a general point, then $\mathbf{R}_X(E)= \ur_X(E)$.
\item \label{item_general_point_in_sigma_r_r}
$\mathbf{R}_X(E)= k$ if and only if the reduced points of $X\cap \BP E$ span $\BP E$.
\item \label{item_measure_rank_by_sigma_r_r}
      For $r\le N+1$,   $\mathbf{R}_X(E) \le r$ if and only if $\exists F\in \sigma^0_{r,r}(X)$
      such that $E \subset F$
      and similarly
      $\ur_X(E) \le r$ if and only if $\exists F\in \sigma_{r,r}(X)$
      such that $E \subset F$.
\item If $E = E' \oplus E''$, then $\mathbf{R}_X(E) \le \mathbf{R}_X(E') + \mathbf{R}_X(E'')$.
\end{enumerate}
\end{proposition}

A corollary of Theorem~\ref{thm:can_calculate_rank_as_superspace_rank} is the following.

\begin{corollary} \label{cor:bound_on_superspace_rank}
The maximum $X$-rank  of $p\in\PP(A \otimes W)$, where $X = Seg(\PP A \times Y)$ and $Y\subset \BP W$ is not contained
in a hyperplane,
is at most $\tdim W$. In particular, $\s_r(X)=\BP (A\ot W)$ for all $r\geq \tdim W$.
\end{corollary}

 For any subvariety $X\subset \PP V$ and $p \in \PP V$,
  one has $\mathbf{R}_X(p) \le \dim V - \dim X$, see  \cite[Prop.~5.1]{LTrank} .
In the case of $X= Seg(\PP A \times Y)$ the upper bound of Corollary~\ref{cor:bound_on_superspace_rank}
  is much smaller.
The proposition below   generalizes \cite[Prop.~5.1]{LTrank} to the case of the  $Y$-rank of a subspace.
If $\dim Y \ge \dim A$, then the bound below is better than the bound of Corollary~\ref{cor:bound_on_superspace_rank}.

\begin{proposition}\label{prop:bound_on_superspace_rank_LT}
  Suppose $Y \subset \PP W$ is an irreducible subvariety not contained in a hyperplane.
  The maximum $X$-rank  of $p\in\PP(A \otimes W)$, where $X = Seg(\PP A \times Y)$
    is at most $\dim W - \dim Y + \dim A -1$
\end{proposition}
The proof of the proposition relies on Lemma~\ref{lem:bound_on_superspace_rank_for_disjoint},
  whose proof is essentially identical to the proof of \cite[Prop.~5.1]{LTrank}. Proofs are given in \S\ref{agproofs}.

Here is  an easy example  where the bounds of Corollary~\ref{cor:bound_on_superspace_rank}
    and Proposition~\ref{prop:bound_on_superspace_rank_LT} are optimal.

\begin{example}
  Let $Y=v_d(\pp 1)$ denote the projectivization of the set of symmetric tensors in $S^d\BC^2$ of rank one,
    and $X=Seg(\BP A\times v_d(\pp 1))$. Let $p=a_1\ot x^d + a_2\ot x^{d-1}y$.
  Then $\mathbf{R}_X(p)=\mathbf{R}_Y(p(A^*))=d+1$ and $\ur_X(p)=\ur_Y(p(A^*))=2$.
\end{example}
\begin{proof}
To see $\ur_Y(p(A^*))=2$, note
that
\[
  p(A^*)=\tlim_{t\ra 0}\langle x^d, (x+ty)^d\rangle.
\]
One has  $\mathbf{R}_Y(p(A^*)) \le d+1$,
 because  $\dim (S^d \CC^2 ) = d+1$.
The assertion $\mathbf{R}_Y(p(A^*)) \ge d+1$  follows from the standard arguments about
  points on the tangential variety to the rational normal curve, see, e.g.~\cite{LTrank}.
More precisely,
  suppose by contradiction $\mathbf{R}_Y(p(A^*)) = d$,
  so that there exists a hyperplane $H \subset \PP (S^d \CC^2)$
  containing $\PP(p(A^*))$ spanned (set-theoretically) by points in $Y\cap H$.
In particular, $Y\cap H$ has at least $d$ distinct points.
But the scheme $Y\cap H$ contains $[x]$ with  multiplicity at least $2$
  and thus the degree of scheme $Y \cap H$ is at least $d+1$.
This is impossible, since the degree of $Y$ is $d$.
\end{proof}

In the special case of tensors we have:
\begin{corollary}\label{tenrkbnd}
The maximum rank of an element of $A_1\otc A_n$, where
$
 \tdim A_1\leq \cdots \leq \tdim A_n
$,
is at most:
\[
 \begin{cases}
    \tdim A_1\cdots \tdim A_{n-1},  \quad  \text{ if } \tdim A_{n} > \tdim A_1 + \cdots  + \tdim A_{n-1},\\
    \tdim A_1\cdots \tdim A_{n-1} - \tdim A_1 - \cdots  - \tdim A_{n-1}  + \tdim A_{n} - 1, \quad  \text{ otherwise.}
 \end{cases}
\]
In particular, the maximum rank of an element of $(\BC^m)^{\ot n}$ is
at most $m^{n-1} - (n-2)m - 1$, compared with the maximum border rank, which is typically $\lfloor\frac{m^n}{(m-1)n+1}\rfloor\sim \frac{m^{n-1}}n $.
\end{corollary}

\begin{corollary}\label{cor:PGL_quotient}
  Let $A$, $W$ be   vector spaces with $\aaa:=\dim A \le \dim W$.
  Consider the rational map $\pi: \BP(A\otimes W) \dashrightarrow G(\aaa, W)$,
  with $\pi(p) := p(A^*)$ whenever $p(A^*)$ has maximal dimension.
  Then:
  \begin{enumerate}
    \item \label{item:pi_is_the_PGL_quotient}
          For any $T \in G(\aaa, W)$ the preimage $\pi\inv([T]) \subset \BP(A\otimes W)$
            is a $PGL(A)$ orbit isomorphic to $PGL(A)$.
    \item $\pi$ is $PGL(W)$ equivariant.
    \item \label{item:sigma_r_is_preimage}
         for $r\ge \dim A$ and any subvariety $Y \subset \BP W$,
         \[
           \sigma_r(Seg(\BP A \times Y)) =  \overline{\pi^{-1}(\sigma_{r,\aaa} (Y))}.
         \]
  \end{enumerate}
\end{corollary}

\begin{remark}
  Although we will not use it in this article,
    we remark that, with notation as in Corollary~\ref{cor:PGL_quotient},
    the pair $(G(\aaa, W), \pi)$ is a geometric quotient
      (in the sense of \cite[Def.~0.6]{MR1304906})
    of the open subset $U \subset \BP(A\otimes W)$, where $\pi$ is well defined, by $PGL(A)$.
  Note $U$ is the set of those $p$ for which $p\colon A^* \to W$ is injective.
  Its complement is $\s_{\aaa-1}(Seg(\BP A\times \BP W))$.
\end{remark}

Often auxiliary varieties one constructs from a given variety $X\subset \BP V$ have an
{\it expected dimension}. That is, sometimes one can estimate the dimension of the auxiliary variety
from the dimension of $X$, such that for a generic variety of such a dimension, the estimate
gives the actual dimension.
If there is a difference between the expected dimension and the actual dimension,
  then we say the auxiliary variety is \emph{defective} and the difference is called the \emph{defect}.

For a variety $X\subset \BP V$, the expected dimension of $\s_{r,k}(X)$ is $\tmin\{rn+k(r-k), \tdim G(k,V)\}$.

The following bound appeared in \cite{ciliberto_cools_Grassmann_secant_extremal_varieties}.

\begin{theorem}\cite{ciliberto_cools_Grassmann_secant_extremal_varieties}\label{thm_estimate_on_dim_of_Gr_secants}
 Let $X^n\subset\pp N$ be an irreducible variety not contained in a hyperplane.
If $\s_{r,k}(X)\neq 
G(k,\BC^{N+1})$,
 then
$$
\tdim \s_{r,k}(X)\geq rn+k(r-k)-(n-1)(r-k).
$$
In particular $\tdim \s_{r,r}(X)=\tmin\{rn, r(N+1-r)\}$ for all irreducible varieties $X$, and when
$X$ is a curve, $\s_{r,k}(X)$ is always of the expected dimension.
\end{theorem}

We conclude with another situation  where  the dimensions of secant varieties to $Seg(\PP A \times X)$
  are possible to determine.

\begin{proposition}\label{cor:dimensions_for_secants_of_product}
  Suppose $r \le \dim A=\aaa $, $r < \dim W = \bw$ and $X\subset \BP W$ is of dimension $n$ and not contained in a  hyperplane.
  \begin{itemize}
   \item   If $\codim_{\BP W} X \ge r-1$,
           then $\sigma_r(Seg(\BP A \times X))$ is of the expected dimension $ r (\aaa + n )-1$
           and $\s_{r,\aaa}(X)$ is of the expected dimension $rn+\aaa(r-\aaa)$.
   \item   If $\codim_{\BP W} X < r-1$,
             then $\sigma_r(Seg(\BP A \times X)) = \sigma_r(Seg (\BP A \times \BP W))$ and
             thus 
             \[
                \dim \sigma_r(Seg(\BP A \times X)) = r( \aaa + \bw - r)-1.
             \]
           Unless $r=\aaa$,
             the secant variety $\sigma_r(Seg(\BP A \times X))$ is defective with defect $r(r - \bw + n - 2)$.
  \end{itemize}
\end{proposition}

Proposition \ref{cor:dimensions_for_secants_of_product} is proved in \S\ref{agproofs}.

\section{Situations where ranks and border ranks of points coincide}\label{coincidesect}

 Among (rational) homogeneous varieties $G/P\subset \BP V$, where $G$ is a complex semi-simple
Lie group and $V$ a $G$-module, there are certain special classes, for example the
{\it generalized cominuscule varieties}, which are the compact Hermitian symmetric spaces in their
homogeneous embeddings. Among these there is a   sub-class, called the
{\it sub-cominuscule varieties} which can be characterized by the fact that the only
$G$-orbits in $\BP V$ are the secant varieties of $G/P$.
There is a short list of such varieties:
 $Seg(\BP A\times\BP B)$,
$v_2(\pp n)$ (the rank one symmetric
$(n+1)\times (n+1)$ matrices), $G(2,n)$  (the rank two skew-symmetric
$n\times n$ matrices), the Cayley plane $\BO\pp 2$,
 and  the $10$-dimensional spinor variety  $\BS_5$,   see \cite{LMseries},  \cite{LMsel}.
Among these, in the case of the first three, the orbits are just the set of matrices
(resp. symmetric, resp. skew-symmetric) of a given rank.

\begin{proposition} \label{rurchss} Let  $X\subset \BP V$ be a sub-cominuscule
variety.   Then for all $p\in \BP V$, $\mathbf{R}_X(p)=\ur_X(p)$.
\end{proposition}
\begin{proof}
  Since $X$-rank and $X$-border rank are invariant under the automorphism group of $X \subset \PP V$
  and the only orbits are the secant varieties, the rank and the border rank must coincide.
\end{proof}

\section{Ranks of  points in $\BC^2\ot \BC^{\bbb}\ot \BC^{\ccc}$}\label{jajasect}

In this section we review and summarize   facts on tensors in $\CC^2 \otimes \CC^{\bbb} \otimes \CC^{\ccc}$,
which are scattered throughout  the literature.
The main advantage of this case   is the existence of Kronecker's normal form,
  which we review in \S\ref{sect_kronecker_normal_form}.
A theorem of Grigoriev, Ja'Ja' and Teichert calculates the rank of a tensor in the Kronecker  normal form.
We recall this theorem and its  proof in \S\ref{sect_thm_jaja}.
If in addition $\bbb \le 3$, then there is only a finite number of $GL_2 \times GL_{\bbb} \times GL_{\ccc}$-orbits in $\CC^2 \otimes \CC^{\bbb} \otimes \CC^{\ccc}$.
We present a description of these orbits, with a representative of each of them
   and   list the ranks of each of these tensors in \S\ref{section_orbits}.

Let $B \simeq \CC^{\bbb}$ and $C \simeq \CC^{\ccc}$.
Throughout this section by \emph{pencil of matrices}, or simply \emph{pencil},
   we mean a two dimensional linear subspace of $B \ot C$.
To denote a pencil,
  it is convenient to use matrix notation so choose bases of $B,C$
  and write the linear subspace as $sM+tN$, where $M,N\in B\ot C$ and $s,t\in \BC$.
Such an expression  is not unique, but each pair of linearly independent $M,N$ determines a pencil.

For a tensor $p \in \CC^2 \otimes B \otimes C$ set $M:=p(a^1)$ and $N:=p(a^2)$, where $a^1, a^2$
  is a basis of $(\CC^2)^*$.
By Corollary~\ref{cor:PGL_quotient}, two tensors $p, p' \in \CC^2 \otimes B \otimes C$
  are equivalent with respect to the action of $GL_2 \times GL(B) \times GL(C)$ if and only if
  the determined pencils or matrices are equivalent with respect to the action of $GL(B) \times GL(C)$.

\subsection{Kronecker's normal form}\label{sect_kronecker_normal_form}

Kronecker determined a normal form for pencils of matrices.
His classification works over arbitrary closed fields,
 but we only present the results over $\BC$.
The result is as follows (see, e.g., \cite[Chap. XII]{MR0107649}):

Define the $\ep\times (\ep+1)$ matrix
$$
L_{\ep}=L_{\ep}(s,t)=\begin{pmatrix}
s        & t & &   \\
  &\ddots    &\ddots &  \\
&         & s        & t
\end{pmatrix}.
$$
\begin{proposition}
Every pencil can be written as:
\be\label{kronnormalform}
sM+tN=
\begin{pmatrix}
L_{\ep_1} &         & & & & &\\
          &\ddots   & & & \\
 & &                &L_{\ep_k} & & & &\\
& & & &  L_{\eta_1}^T  & & &\\
& & & &   & \ddots & &\\
& & & &  & & L_{\eta_l}^T &\\
& & & &  & &  &s\Id_f +t F
\end{pmatrix},
\ene
where $F$ is an $f\times f$ matrix in
Jordan normal form (one can also use rational canonical form) and $T$ denotes the transpose.
\end{proposition}

The normal form above is not always unique.
Say $F$ has Jordan blocks, $F_{i,j}$ where $\l_i$ is the $i$-th eigenvalue.
If there is no block of the form $L_{\ep_i}$ or $L_{\eta_j}^T$,
  then we may assume at least one of the $\l_i$ is zero by changing basis in $\BC^2$.
In general, one can always change the bases in $\BC^2$, $B$ and $C$
  to obtain at least one of the $\l_i=0$.
We will not use this in general, so we omit the proof, but we illustrate a  non-trivial  case in
  the proof of Proposition~\ref{prop_orbits_234}. The  general case is not much different than the illustrated case.
If the blocks $F_{i,j}$ are such that there are no $1$'s above the diagonal, then
  we can also normalize one of the $\l_i=1$ by rescaling $t$.

\begin{example}
Suppose $f=3$. The possible Jordan normal forms of $3\times 3$ matrices are
$$
\begin{pmatrix} \l & & \\ &\mu& \\ & &\nu\end{pmatrix},\
\begin{pmatrix} \l & & \\ &\l& \\ & &\mu\end{pmatrix},\
\begin{pmatrix} \l &  1 & \\ &\l& \\ & &\mu\end{pmatrix},\
\begin{pmatrix} \l & & \\ &\l& \\ & &\l\end{pmatrix},\
\begin{pmatrix} \l &1 & \\ &\l& \\ & &\l\end{pmatrix},\
\begin{pmatrix} \l & 1 & \\ &\l&1 \\ & &\l\end{pmatrix}.
$$
Now suppose $\langle M,N  \rangle$ is a pencil of $3\times 3$ matrices in its Kronecker normal form,
which has no block of the form $L_{\ep_i}$ or $L_{\eta_j}^T$ (thus it only has the block $s\Id_f +t F$).
Then it can respectively be normalized to
\be\label{3b3forms}
\begin{pmatrix} 0 & & \\ &1& \\ & &-1\end{pmatrix},\
\begin{pmatrix} 0 & & \\ &0& \\ & &1\end{pmatrix},\
\begin{pmatrix} 0 &  1 & \\ &0& \\ & &1\end{pmatrix},\
\begin{pmatrix} 0 & & \\ &0& \\ & &0\end{pmatrix},\
\begin{pmatrix} 0 &1 & \\ &0& \\ & &0\end{pmatrix},\
\begin{pmatrix} 0 & 1 & \\ &0&1 \\ & &0\end{pmatrix}
\ene
\end{example}
Note that the fourth case is not a pencil.
The first case requires explanation --- we claim that all pencils of the form:
\[
  s \begin{pmatrix}  1 & & \\ &  1& \\ & &  1\end{pmatrix}
+ t \begin{pmatrix} \l & & \\ &\mu& \\ & &\nu\end{pmatrix}
\]
  where $\lambda, \mu, \nu$ are distinct, are equivalent. In particular, any such  is equivalent
to one where $\lambda =0, \mu =1, \nu =-1$.
To prove the claim, first get rid of $\lambda$  by replacing $s$ with $s_1:=s+\lambda t$:
\[
  s_1 \begin{pmatrix}  1 & & \\ &  1& \\ & &  1\end{pmatrix}
+ t \begin{pmatrix} 0 & & \\ &\mu_1& \\ & &\nu_1\end{pmatrix}.
\]
Note that $0$, $\mu_1$ and $\nu_1$ are still distinct.
Next we   replace $t$ with $t_2: = s_1 + t \mu_1$:
\[
  s_1 \begin{pmatrix}  1 & & \\ &  0& \\ & &\mu_2\end{pmatrix}
+ t_2 \begin{pmatrix} 0 & & \\ &1& \\ & &\nu_2\end{pmatrix}.
\]
where $\mu_2 = 1- \frac{\nu_1}{\mu_1}$ and $\nu_2 = - \frac{\nu_1}{\mu_1}$
 and $0$, $\mu_2$ and $\nu_2$ are distinct.
Then we transport the constants to the first 2 entries by setting
$s_3:= \frac{1}{\mu_2} s_1$ and $t_3:= \frac{1}{\nu_2} t_2$:
\[
  s_3 \begin{pmatrix}  \frac{1}{\mu_2} & & \\ &  0& \\ & &1\end{pmatrix}
+ t_3 \begin{pmatrix} 0 & & \\ &\frac{1}{\nu_2}& \\ & &1\end{pmatrix}.
\]
It only remains to change basis in $B$ by sending $b_1$ to $\frac{1}{\mu_2} b_1$ and $b_2$ to $\frac{1}{\nu_2} b_2$
to show the pencil is equivalent to:
\[
  s_3 \begin{pmatrix}  1 & & \\ &  0& \\ & &1\end{pmatrix}
+ t_3 \begin{pmatrix} 0 & & \\ &   1& \\ & &1\end{pmatrix}.
\]
Thus every two such pencils are equivalent.

If $F$ is $4\times 4$ it is no longer possible to normalize   all the constants  in the case
$$
\begin{pmatrix}
\l & 1& &  \\ &\l & & \\ & & \mu & 1 \\ & & & \mu
\end{pmatrix}.
$$
Essentially because of this, the only spaces of tensors
$\BC^{\aaa}\ot\BC^{\bbb}\ot \BC^{\ccc}$, $2\leq \aaa\leq\bbb\leq \ccc$, that   have a finite
number of $GL_{\aaa}\times GL_{\bbb}\times GL_{\ccc}$-orbits,
are $\BC^2\ot \BC^2\ot \BC^{\ccc}$   and $\BC^2\ot \BC^3\ot \BC^{\ccc}$ (see \cite{Kacorb,KWorbit}).
Moreover,  in these cases,
  any tensor lies in a $\BC^2\ot \BC^2\ot \BC^4$ in the first case and a $\BC^2\ot \BC^3\ot \BC^6$ in
the second for some linear subspace $\BC^4 \subset \BC^{\ccc}$ or $\BC^6 \subset \BC^{\ccc}$, respectively.

\subsection{Theorem of Grigoriev,   Ja'Ja', and Teichert}\label{sect_thm_jaja}

The Kronecker normal form is convenient for calculating the  $Y$-rank of a pencil of matrices,
  where  $Y:=Seg(\PP B \times \PP C)$.
It turns out that the contribution of each block to the $Y$-rank is separate,
  that is the rank of the pencil is a sum of the ranks of each blocks in \eqref{kronnormalform},
  see Proposition~\ref{jajaredlem}.
The rank of each block is obtained in Lemmas~\ref{companlem}, \ref{leplem}, \ref{jordbllem}.
Theorem \ref{jajathm} below summarizes these calculations.

For a fixed linear map $F: \BC^f\ra \BC^f$, let $d(\l)$ denote the number of Jordan blocks
of size at least two associated to the eigenvalue $\l$, and let $m(F)$ denote
the maximum of the $d(\l)$.

\begin{theorem}[Grigoriev,   Ja'Ja', Teichert] \cite{MR519843,MR521052,Teichert}\label{jajathm}
A pencil  of the form \eqref{kronnormalform} has rank
$$\sum_{i=1}^k (\ep_i+1)+\sum_{j=1}^l (\eta_j+1)+f+  m(F).
$$
In particular, the maximum possible  rank of a tensor in $\BC^2\ot \BC^{\bbb}\ot \BC^{\bbb}$ is
$\lfloor \frac{3\bbb}2\rfloor$.
\end{theorem}

For $\bbb =2n$, the maximum possible rank is obtained by a pencil of the form \eqref{equ:tensor_with_maximal_rank} below.
For $\bbb =2n+1$, take $s\Id_{2n+1} + t N$, where $N$ is as in \eqref{equ:tensor_with_maximal_rank},
   but viewed as a $(2n+1) \times (2n+1)$ matrix --- just add a row and a column of zeros.

\begin{remark}
In \cite{MR521052} Theorem \ref{jajathm} is stated as an inequality (Cor. 2.4.3 and Thm. 3.3), but
the results are valid over arbitrary closed fields. In \cite{MR519843} the results are stated,
but not proved, and the reader is referred to \cite{MR0453768} for indications towards the proofs. In
\cite{BCS} a complete proof is given of an equivalent statement in terms of the elementary divisors
of the pair,  and the text states the proof is taken from the
unpublished PhD thesis \cite{Teichert}.
\end{remark}

\begin{proposition}\label{jajaredlem}
Let $B=B_1\op B_2$, $C=C_1\op C_2$,
  $p_1\in A \ot B_1\ot C_1$, $p_2\in A \ot B_2\ot C_2$, and $p_3\in A \ot B_2\ot C_1$.
Suppose $\bbb_i$ and $\ccc_i$ are the dimensions of, respectively, $B_i, C_i$ for $i =1,2$.
Then
\begin{enumerate}
  \item \label{item:bound_on_rank_of_sum_p1_p2_p3}
     If $p_2: B_2^*\ra  A \ot C_2$ is injective, then $\mathbf{R}(p_1+p_2+p_3)\geq \mathbf{R}(p_1)+ \bbb_2$.
  \item \label{item:equality_for_rank_of_sum_p1_p2}
     If both maps $p_2: B_2^*\ra A \ot C_2$ and $p_2: C_2^*\ra A \ot B_2$ are injective
       and $\mathbf{R}(p_2)=\tmax\{ \bbb_2,\ccc_2\}$ (the minimum possible for such $p_2$), then
       $\mathbf{R}(p_1+p_2)= \mathbf{R}(p_1)+\mathbf{R}(p_2)$.
\end{enumerate}
\end{proposition}

\begin{proof}
  To prove \ref{item:bound_on_rank_of_sum_p1_p2_p3} let
  $p := p_1+ p_2 + p_3$, $r:=\mathbf{R}(p)$ and write $p = \sum_{i=1}^r a_i\ot b_i\ot c_i$ in some minimal presentation.
  Consider the projection $\rho \colon B \to  B_2$.
  Since $p|_{B_2^*}$ is injective,
    we may assume $\rho(b_1)\hd \rho(b_{\bbb_2})$ form a basis of $B_2$.
  Let $B_2' \subset B$ be the span of $b_1 \hd b_{\bbb_2}$.
  Note that the composition $B_1 \hookrightarrow B \to B /B_2'$ is an isomorphism.
  Consider the following composed projection $\pi$:
  \[
    A \otimes B \otimes C \to A \otimes (B/{B_2'})  \otimes C  \to A \otimes (B/{B_2'})  \otimes C_1 .
  \]
  The kernel of $\pi$ contains $A \otimes B\otimes C_2$,
    and  $\pi|_{A\otimes B_1 \otimes C_1}$ is an isomorphism.
  Thus $\pi(p)$ is $p_1$ (up to the isomorphism $B /B_2' \simeq B_1$)
    and also $\pi(p) = \sum_{i=\bbb_2 +1}^r \pi(a_i\ot b_i\ot c_i)$.
  Hence $\mathbf{R} (p_1) \le r - \bbb_2$ as claimed in \ref{item:bound_on_rank_of_sum_p1_p2_p3}.

  Statement~\ref{item:equality_for_rank_of_sum_p1_p2}
    follows from \ref{item:bound_on_rank_of_sum_p1_p2_p3} with $p_3=0$ used twice
    (once with the roles of $B$ and $C$ exchanged)
      to note $\mathbf{R}(p_1+p_2)\geq \mathbf{R}(p_1)+\bbb_2$ and $\mathbf{R}(p_1+p_2)\geq \mathbf{R}(p_1)+\ccc_2$,
    and the inequality $\mathbf{R}(p_1+p_2)\leq \mathbf{R}(p_1)+\mathbf{R}(p_2)$.
\end{proof}

Proposition  \ref{jajaredlem}  was stated  and proved
for the special case $\tdim A=2$ in  \cite[Lemma 19.6]{BCS}.
The lemma  is worth generalizing
because it provides    an example of a situation where the
{\it additivity} conjectured by Strassen \cite{MR0521168} holds.

  A generic $\bbb\times \bbb$ pencil is diagonalizable (as the conditions to have repeated eigenvalues or
bounded rank are closed conditions)  and thus of rank $\bbb$.
Thus for most (more precisely, a Zariski open subset of) pencils that are not diagonalizable, 
  a perturbation by a general rank one matrix will make it diagonalizable.
The next lemma shows that such a perturbation can be achieved 
  if the pencil is a Jordan pencil $\langle Id_{\bbb},F\rangle$ 
  with each eigenvalue
  associated  to just one Jordan block.

\begin{lemma}\label{companlem}
   Let $p = \langle \Id_{\bbb}, F \rangle$ with $F$ a size $\bbb$  matrix
     in Jordan normal form with no eigenvalue having more than one associated Jordan block.
   Then  $\mathbf{R}(p)=\bbb+1$ if the Jordan form is not diagonal,
      and $\mathbf{R}(p)=\bbb$ if the Jordan form is diagonal.
\end{lemma}
An equivalent lemma is proved in \cite[Prop.~19.8]{BCS}. 
We present a slightly different proof in Section~\ref{parsymsect},
 where we prove a more general Lemma~\ref{lem_rank_of_symmetric_pencil_with_distinct_eigenvalues}.

\begin{lemma}\label{leplem}
  Let $p_2\in \CC^2 \otimes \CC^{\epsilon} \otimes \CC^{\epsilon+1}$ be a tensor,
    whose Kronecker normal form is $L_{\ep}$.
  Then  $\mathbf{R}(p_2)=\ep+1$.
  In particular, Proposition \ref{jajaredlem}\ref{item:equality_for_rank_of_sum_p1_p2} applies for $p_2$.
  Analogous statements for tensor whose Kronecker normal form is 
  $L_{\eta}^T$
  are also true.
\end{lemma}
This lemma is proved in \cite[Prop.~19.9]{BCS}. 
We present a different approach, which allows us to present a form for all minimal decompositions $p_2$,
   see Remark~\ref{rem_any_decomposition_of_L_epsilon}.

\begin{proof}
  Consider $p_2: (\CC^{\epsilon+1})^* \to \CC^2 \otimes \CC^{\epsilon}$ 
    and define $E:= p_2\bigl((\CC^{\epsilon+1})^*\bigr) \subset \CC^2 \otimes \CC^{\epsilon}$.
  It is parametrized by
  \begin{equation}\label{equ_matrix_for_L_epsilon}
     \begin{pmatrix}
        \gamma_0 & \gamma_1 & \gamma_2 & \dots & \gamma_{\epsilon-2} & \gamma_{\epsilon-1}\\ 
        \gamma_1 & \gamma_2 & \gamma_3 & \dots & \gamma_{\epsilon-1} & \gamma_{\epsilon}
     \end{pmatrix}.
  \end{equation}
  In particular $\dim E = \epsilon+1$,
    so by Proposition~\ref{elemrank}\ref{item_general_point_in_sigma_r_r},
    we have $\mathbf{R}(p_2) = \epsilon+1$ if and only if $\PP E$
    is spanned by 
reduced points of $X:=Seg(\PP^1 \times \PP^{\epsilon-1})$.
  The intersection $Z:=X \cap \PP E$ is defined by 
    the $2\times 2$ minors of \eqref{equ_matrix_for_L_epsilon}. 
  These equations define a rational normal curve of degree $\epsilon$ in $\PP E$.
  In particular, $Z$ is reduced and its points span $\PP E$.
\end{proof}

\begin{remark}\label{rem_any_decomposition_of_L_epsilon}
   The proof above shows how to obtain any decomposition of $p_2$ as a sum of $\epsilon+1$
     simple tensors.
   Pick $\fromto{[x_0,y_0]}{[x_{\epsilon}, y_{\epsilon}]}$ to be pairwise distinct points on $\PP^1 \simeq Z$.
   Each of these points gives rise to a rank $1$ matrix (defined up to scale):
   \[
   \begin{pmatrix}
     {x_i}^{\epsilon} & {x_i}^{\epsilon-1}{y_i} & \dots & {x_i}{y_i}^{\epsilon-1}\\ 
     {x_i}^{\epsilon-1}{y_i} & {x_i}^{\epsilon-2}{y_i}^2 & \dots & {y_i}^{\epsilon}
   \end{pmatrix} =
   \begin{pmatrix}
      {x_i} \\ {y_i}
   \end{pmatrix}
   \begin{pmatrix}
      {x_i}^{\epsilon-1} & {x_i}^{\epsilon-2} y_i & \dots & {y_i}^{\epsilon-1}
   \end{pmatrix}.
   \]
   These matrices for $i \in \setfromto{0}{\epsilon}$ span $E$. 
   For example, if $\epsilon=3$, then the quadruple $[1,0]$, $[0,1]$, $[1,1]$, $[1,-1]$ gives a decomposition of 
   \[
      p_2 = a_1\otimes b_1\otimes c_1 + (a_1\otimes b_2+a_2\otimes b_1)\otimes c_2 
                + (a_1\otimes b_3+a_2\otimes b_2)\otimes c_3 + a_2\otimes b_3\otimes c_4
   \]
   into the following four simple summands:
\begin{align*}
   p_2 & = a_1\otimes b_1\otimes (c_1-c_3)\\ 
  &+ a_2\otimes b_3\otimes (c_4-c_2) \\
& + \frac 1 2 (a_1+a_2)\otimes (b_1+b_2+b_3)\otimes (c_2+c_3) \\
&  + \frac 1 2 (a_1-a_2)\otimes (b_1-b_2+b_3)\otimes (c_3-c_2).
\end{align*}
\end{remark}

\begin{lemma}[{\cite[Prop. 19.10]{BCS}}]\label{jordbllem}
  Let a pencil $E$ be given by $\langle \Id_{2n}, F \rangle$ with $F$ a matrix consisting
    of $n$ Jordan blocks of size two, all  with the same eigenvalue. Then
    $\mathbf{R}(E)=3n$.
\end{lemma}
\begin{proof}
  After possibly changing bases, we may write   $E  =\langle M, N \rangle$, where
  \begin{equation}\label{equ:tensor_with_maximal_rank}
     sM + tN = \begin{pmatrix}
                 s \Id_n  & t \Id_n \\
                       0  & s \Id_n
               \end{pmatrix}.
  \end{equation}
  Let $B = B_1 \oplus B_2$ and $C = C_1 \oplus C_2$ be the splitting corresponding to these blocks,
     with $B_1 = B_2 = {C_1}^* = {C_2}^* = \CC^n$
     and
  \[
     M =  \Id_{{C_1}^*,B_1} + \Id_{{C_2}^*, B_2} \in B_1 \otimes C_1 \oplus B_2 \otimes C_2, \quad
     N = \Id_{{C_2}^*, B_1} \in B_1 \otimes C_2.
  \]
  Here $\Id_{{C_i}^*, B_j}$ is a distinguished isomorphism ${C_i}^* = B_j$, whose matrix is the identity matrix.
  Let~$A \simeq \CC^2$.
  Consider the tensor in $A\otimes B \otimes C$ corresponding to $E$,
    that is $p = s \otimes M + t \otimes N$, where we think of $s,t$ as a basis of $A$.
  Suppose $\mathbf{R}(E)=\mathbf{R}(p) = r$.
  Clearly $r\le 3n$.
  Write a minimal decomposition $p = \sum_{i=1}^r a_i \otimes b_i \otimes c_i$.
  Let $\rho: B \to B_2$ be the projection with kernel $B_1$.
  Since the map $p\circ \rho^*: B_2^* \to A \otimes C$ is injective,
    we may choose a basis of $B_2$ out of the set $\setfromto{\rho(b_1)}{\rho(b_r)}$.
  Without loss of generality, suppose $\setfromto{\rho(b_1)}{\rho(b_n)}$ is a basis,
    and let $B_2'$ be the span of $\setfromto{b_1}{b_n}$.
  Consider the composition $\pi$:
  \[
    C^* \stackrel{p}{\to} A \otimes B \to A \otimes (B/B_2')
  \]
  where the second map is the natural projection.
  If $\pi$ is written as a tensor, then
  \[
    \pi = \sum_{i=n+1}^r a_i \otimes (b_i \mod B_2') \otimes c_i.
  \]
  Thus it suffices to prove that $R(\pi) = 2n$,
    which is equivalent to say that $\pi :C^* \to A \otimes (B/B_2')$ is injective.
  Suppose $\gamma \in C^*$ is in the kernel of  $\pi$.
  Then $p(\gamma) \in A \otimes B_2'$.
  Decompose $\gamma = \gamma_1 + \gamma_2$ with respect to $ C^* = {C_1}^* \oplus {C_2}^*$.
  Now
    \[
      p(\gamma) =  s \otimes (\Id_{{C_1}^*,B_1} + \Id_{{C_2}^*, B_2}) (\gamma) + t \otimes  \Id_{{C_2}^*, B_1} (\gamma)
                =  s \otimes \gamma^{B_1}_1 + s\otimes \gamma^{B_2}_2  + t \otimes \gamma^{B_1}_2
    \]
  where $\gamma^{B_j}_i := \Id_{{C_i}^*,B_j}(\gamma_i) \in B_j$.
  Since  $p(\gamma) \in A \otimes B_2'$ and $B_1 \cap B_2' =0$,
    we must have $t \otimes \gamma^{B_1}_2 = 0$, so $\gamma_2=0$.
  Therefore  $p(\gamma) = s \otimes \gamma^{B_1}_1 = 0$ for the same reason.
  Thus $\gamma =0$ and $\pi$ is injective as claimed.
\end{proof}

We are now ready to prove the theorem of Grigoriev, Ja'Ja', Teichert:
\begin{proof}[Proof of Theorem~\ref{jajathm}]
    Let $p$ be as in \eqref{kronnormalform}.
    First observe, that by Lemma~\ref{leplem} and Proposition~\ref{jajaredlem}\ref{item:equality_for_rank_of_sum_p1_p2}:
    \[
       \mathbf{R}(p) = \sum_{i=1}^k \mathbf{R}(L_{\epsilon_i}) + \sum_{j=1}^l \mathbf{R}(L^T_{\eta_j}) 
                       + \mathbf{R}(s \Id_f + t F)
                     = \sum_{i=1}^k (\epsilon_i + 1) + \sum_{j=1}^l (\eta_j+1) + \mathbf{R}(s \Id_f + t F)
    \]
    Thus it is sufficient to prove the theorem for $p = s \Id_f + t F$.

    Reordering the Jordan blocks 
       we can write $F = \begin{pmatrix}
                          F' & 0 \\ 
                          0  & D 
                      \end{pmatrix}$
       where $D$ is a diagonal matrix and $F'$ is a $f' \times f'$ matrix with only Jordan blocks of size at least two.
    The rank of  $(s \Id_{f-f'} + t D)$ is \mbox{$(f - f')$} by Lemma~\ref{companlem},
       thus we can apply Proposition~\ref{jajaredlem}\ref{item:equality_for_rank_of_sum_p1_p2} 
       and 
 \[
    \mathbf{R}(s \Id_f + t F) = \mathbf{R}(s \Id_{f'} + t F') + (f-f').
 \]
    Thus from now on assume $F$ has only Jordan blocks of size at least two, i.e., $F= F'$.
   
    The statement of Theorem claims  $\mathbf{R}(p) \le f+m(F)$, where $m(F) = \max_{\lambda \in \CC} (d(\lambda))$,
       and $d(\lambda)$ is the number of Jordan blocks of $F$ with eigenvalue $\lambda$.
    To obtain the upper bound $\mathbf{R}(p) \le f+m(F)$, 
       divide the Jordan blocks into $m(F)$ groups $\fromto{F_1}{F_{m(F)}}$, 
       with (after reordering the Jordan blocks) 
       $F = \begin{pmatrix}
                          F_1 &  \\ 
                            &  \ddots & \\
                            &         & F_{m(F)}
                      \end{pmatrix}$
      and such that in each $F_{\alpha}$  there is at most one block with given eigenvalue.
    Then $\mathbf{R}(p) \le \sum_{\alpha =1}^{m(F)} \mathbf{R}(F_{\alpha}) = f + m(F)$ 
      by Lemma~\ref{companlem}.
    
    To obtain the lower bound $\mathbf{R}(p) \ge f+m(F)$,
      let $\lambda$ be an eigenvalue of $F$ that appears in $m(F)$ Jordan blocks of $F$.
    Let $G_1$ be a matrix with $m(F)$ Jordan blocks of size two with eigenvalue $\lambda$.
    Then by reordering variables we can write 
                      $F = \begin{pmatrix}
                              G_1 & G_3 \\ 
                               0  & G_2 
                           \end{pmatrix}$  
    with $G_2$ a matrix in a Jordan form.
    Proposition~\ref{jajaredlem}\ref{item:bound_on_rank_of_sum_p1_p2_p3} applies for 
      $p = p_1 +p_2+p_3$, where $p_1= s\Id_{2m(F)}+ t G_1$, $p_2= s\Id_{f-2m(F)}+ t G_2$ 
      and $p_3 = t G_3$, and thus:
    \[
       \mathbf{R}(p) \ge \mathbf{R}(p_1) + (f-2m(F)) \stackrel{\text{by Lem.~\ref{jordbllem}}}{=} 3m(F)  + (f-2m(F)) = f+m(F).
    \]
\end{proof}

\section{Orbits}\label{section_orbits}

In this section, for spaces of tensors $A\ot B\ot C$ with a finite number of
$GL(A) \times GL(B) \times GL(C)$-orbits, we present the list of orbits with
their Kronecker normal form (which appeared in \cite{MR1830474}), geometric
descriptions of the orbit closures along with their dimensions, and the ranks and border ranks of the points in the
orbits. These geometric explanations are new to our knowledge.

\begin{remark}
We present orbits in  projective space,  whereas in \cite{MR1830474} they are presented in  affine space,
 so in \cite{MR1830474} there is   one more orbit in each space corresponding to $0$.
\end{remark}

We begin  with the case  $\dim A = \dim B = 2$ and $\dim C =\ccc$.
  Table~\ref{table_orbits_22c}   lists a representative
of each orbit  of the $GL(A) \times GL(B) \times GL(C)$-action  on $\BP( A\ot B\ot C)$,
  where $\dim A = \dim B = 2$ and $\dim C =\ccc$.
Here and in what follows
$$X=Seg(\BP A\times \BP B\times \BP C).
$$
\begin{table}[htb]
$$
\begin{array}{|r|c|c|ll|c|c|}
\hline
\hline
\#&\text{orbit closure}& \dim &\text{Kronecker normal form} & \text{pencil}  &  \ur &  \mathbf{R}  \\
\hline
\hline
1&X& {\ccc}+1& a_1\ot b_1\ot c_1
 & \left(\begin{smallmatrix} s \end{smallmatrix} \right)
&1&1
\\
\hline
2&Sub_{221} &{\ccc}+2&   a_1\ot b_1\ot c_1+a_2\ot b_2\ot c_1
 & \left(\begin{smallmatrix} s & t\end{smallmatrix} \right)
&2&2
\\
\hline
3&Sub_{122} &2{\ccc}&    a_1\ot b_1\ot c_1+ a_1\ot b_2\ot c_2
 & \left(\begin{smallmatrix} s & \\ & s \end{smallmatrix} \right)
&2&2
\\
\hline
4&Sub_{212} &2{\ccc}&   a_1\ot b_1\ot c_1+ a_2\ot b_1\ot c_2
 & \left(\begin{smallmatrix} s \\ t \end{smallmatrix} \right)
&2&2
\\
\hline
5&\tau(X) &2{\ccc}+2&   a_1\ot (b_1\ot c_1+    b_2\ot c_2)+ a_2\ot b_1\ot c_2
 & \left(\begin{smallmatrix} s&t\\ & s \end{smallmatrix} \right)&2&3
\\
\hline
6&\s_2(X)=Sub_{222} &2{\ccc}+3& a_1\ot b_1\ot c_1+ a_2\ot b_2\ot c_2
 & \left(\begin{smallmatrix}  s& \\ & t \end{smallmatrix} \right)
&2&2
\\
\hline
7&\xduals *  &3{\ccc}+1&a_1\ot (b_1\ot c_1+b_2\ot c_3) +  a_2\ot b_1\ot c_2
 & \left(\begin{smallmatrix}  s&t& \\ & & s \end{smallmatrix} \right)
&3&3
 \\
\hline
8&\s_3(X) &3{\ccc}+2&a_1\ot (b_1\ot c_1+   b_2\ot c_2)+ a_2\ot (b_1\ot c_2+b_2\ot c_3)
 & \left(\begin{smallmatrix} s&t& \\  &s& t \end{smallmatrix} \right)
&3&3
\\
\hline
9&\BP(A\ot B\ot C) &4{\ccc}-1&a_1\ot (b_1\ot c_1 + b_2\ot c_3 )
+a_2\ot (b_1\ot c_2+b_2\ot c_4)
& \left(\begin{smallmatrix} s&t&& \\ & &s& t \end{smallmatrix} \right)
   &4&4
\\
\hline
\hline
\end{array}
$$
\caption{\textbf{Orbits in $\BC^2\ot \BC^2\ot \BC^{\ccc}$.}
Each orbit is uniquely determined  by its closure, which is an algebraic variety listed in the second column.
The orbit itself is an open dense subset of this variety.
The dimension of the algebraic variety is in the third column.
The fourth column is the normal form of the underlying tensor,
   the distinct variables are assumed to be linearly independent.
The normal form is also given as a pencil, except the cases of $1$ and $3$, which are not pencils of matrices.
The border rank and rank are given in the next columns.
If ${\ccc}=3$ then $\s_3(X)=\BP( A\ot B\ot C)$, and   case $9$ does not occur.
}\label{table_orbits_22c}
\end{table}

The \emph{subspace variety} $Sub_{ijk}\subset \BP(A\ot B\ot C)$ is the set of tensors $[p]\in \BP (A\ot B\ot C)$ such that
there exists linear subspaces $A'\subset A$, $B'\subset B$, $C'\subset C$ respectively of dimensions
$i,j,k$ such that $p\in A'\ot B'\ot C'$. In other words,
  $Sub_{ijk}$ is the projectivization of the image of the vector bundle
$\cE:=\cS_{G(i,A)}\ot \cS_{G(j,B)}\ot \cS_{G(k,C)}\ra G(i,A)\times G(j,B)\times G(k,C)$
in $A\ot B\ot C$, where $\cS_{G(l,V)}\ra G(l,V)$ is the tautological vector bundle
whose fiber over the point $E$ is the linear space $E$.
Note that if $ k > ij:=l$, then $Sub_{ijk} = Sub_{ijl}$.
So we can always assume $k \le ij$ and similarly for permuted statements.
Then the parameterizing map $\cE \to A \ot B \ot C$ is birational onto its image,
  because for general $p \in Sub_{ijk}$ there is a unique  $A'\subset A$, $B'\subset B$, $C'\subset C$ respectively of dimensions $i,j,k$ such that $p \in A' \otimes B' \otimes C'$.
From this description one computes:
\[
  \dim Sub_{ijk} = \dim \PP (\cE) =  i (\aaa- i) + j (\bbb -j) + k(\ccc - k) + ijk -1.
\]
The other interpretations are as follows:
$\tau(X)$ is the \emph{tangential variety} to the
Segre variety,
$X_*\subset \BP (A^*\ot B^*\ot C^*)$ is the Segre variety in the
dual projective space,  and $ \xduals *\subset \BP (A\ot B\ot C)$ is its dual variety.

The point of $\tau(X)$ is tangent to the point $[a_1\ot b_2\ot c_1]$, the point of
$\xduals *$ contains the tangent plane to the $(\ccc - 3)$-parameter family of points
$[a_2^*\ot b_1^*\ot (s_2c_2^*+s_2c_4^*+s_2c_5^*\cdots + s_{\ccc}c_{\ccc}^*)]$,
 where $(a_j^*)$ is the dual basis to
$(a_j)$ of $A$ etc..
The dual variety $\xduals *$ is degenerate  (i.e., not a hypersurface) except when ${\ccc}\leq 3$,
see, e.g. \cite[p. 46 Cor. 5.10]{GKZ}.

To see the geometric explanations of the orbit closures: Cases 1,6,8 are   clearly on the
respective secant varieties. Cases 2,3,4 are all clearly on the respective subspace varieties, and
it is straightforward to check cases 5 and 7 are tangent to the points asserted. Finally to see
the orbit closures of these points are the stated geometric objects, one can compute the dimensions
of the Lie algebras of their stabilizers to determine the dimensions of their orbit closures and note that
they agree with the dimensions of the geometric objects.

Note $\s_3(X)=\s_3(Seg(\BP(A\ot B)\times \BP C))$ which causes it to
be degenerate with defect three.

The orbits $1$--$8$ are inherited from the ${\ccc}=3$ case,
  in the sense  that they are contained in  $Sub_{223}$.
Orbit $9$ is inherited from the ${\ccc}=4$ case.

\begin{proposition}
   If $[p] \in \PP(A\ot B\ot C)$, with $A \simeq \CC^2$,  $B \simeq \CC^2$, $C \simeq \CC^{\ccc}$,
     then $p$ is in precisely  one of the orbits $1$--$9$ from Table~\ref{table_orbits_22c}.
   The rank and border rank of $[p]$ are  as indicated in the table.
\end{proposition}
\begin{proof}
  Consider  $p: A^* \to B\otimes C$.
  If $\dim p(A^*) =1$, then let $e \in p(A^*)$ be a nonzero element.
  Since $\dim B =2$, the rank of $e$ is one or two, giving the cases $1$ and $3$, respectively.

  Otherwise, $\dim p(A^*) =2$
     and the Kronecker normal from \eqref{kronnormalform} gives   the following cases:
  \renewcommand{\theenumi}{\arabic{enumi}}
  \renewcommand{\labelenumi}{\theenumi.}
  \begin{enumerate}
   \addtocounter{enumi}{1}
   \item There is only one block of the form $L_1$.
   \addtocounter{enumi}{1}
   \item There is only one block of the form $L_1^T$.
   \addtocounter{enumi}{3}
   \item There is only one block of the form $L_2$.
   \item There are two blocks, both of the form $L_1$.
   \addtocounter{enumi}{-3}
   \item There is one block $L_1$ and $F$ is a $1\times 1$ matrix.
         The pencil is then
            $\left(\begin{smallmatrix}  s&t& \\ & & s + \lambda t \end{smallmatrix} \right)$
         and we can normalize   $\lambda$ to zero
         by changing coordinates: $s':= s+\lambda t$ and $c_1' = c_1 + \lambda c_2$.
    \renewcommand{\theenumi}{5--6}
   \item Otherwise, there is no block of the form $L_{\epsilon}$ or  $L_{\eta}^T$
           and $F$ is a $2\times 2$ matrix.
         We can normalize   one of the eigenvalues to $0$.
         We continue, depending on the Jordan normal form of $F$:
   \renewcommand{\theenumi}{\arabic{enumi}}
   \addtocounter{enumi}{-4}
   \item $F=\left(\begin{smallmatrix}  0 & 1 & \\ & 0 \end{smallmatrix} \right)$.
   \item $F=\left(\begin{smallmatrix}  0 &  & \\ & \lambda \end{smallmatrix} \right)$;
         Note that $\lambda \ne 0$, because $\dim p(A^*) =2$.
         Changing the coordinates $t':= \lambda t + s$
         we obtain the pencil $\left(\begin{smallmatrix}  s &  & \\ & t' \end{smallmatrix} \right)$.
  \end{enumerate}
  \renewcommand{\theenumi}{(\roman{enumi})}
  \renewcommand{\labelenumi}{theenumi}

   The ranks are calculated using Theorem~\ref{jajathm}.
   It remains to calculate $\ur(p)$.
The border rank in cases $1$--$4$ and $6$ follow  because
 $\ur(p) \le \mathbf{R}(p)$ and $\ur(p)=1$ if and only if $[p] \in X$.
   Case $5$ is clear too, as the tangential variety is contained in $\sigma_2(X)$.
   $\xduals *$ cannot be contained in $\sigma_2(X)$, as its dimension is larger,
   so for $p\in \xduals * $, we have $2 < \ur (p) \le  \mathbf{R}(p) =3$ proving  case $7$.
   Case $8$ is clear, and case $9$   follows from the dimension count.
\end{proof}

\begin{table}[htb]
$$
\begin{array}{|c|c|c|}
\hline
\#&\text{orbit closure}& \dim  \\
\hline
1&X=Seg(\BP A\times \BP B\times \BP C)& \ccc +2 \\
2& Sub_{221} &\ccc +4 \\
3& Sub_{212} & 2\ccc +1 \\
4& Sub_{122} & 2\ccc +2 \\
5&\tau(X)      & 2\ccc +4 \\
6&  Sub_{222}=\s_2(X) &2\ccc +5 \\
7&\segduals\subset Sub_{223} &3\ccc +3 \\
8&  Sub_{223} &3\ccc +4 \\
9& Sub_{224} & 4\ccc +1  \\
\hline
\end{array}
$$
\caption{The orbits listed in Table~\ref{table_orbits_22c},
          viewed as orbits in $\BC^2\ot \BC^3\ot \BC^{\ccc}$.
         Case $9$ does not occur for $\ccc=3$.}\label{table_orbits_inherited_22c}
\end{table}

Now  suppose  $\dim A =2$, $\dim B = 3$ and $\dim C =\ccc$.
The list of orbits for $\ccc=3$ with their Kronecker normal forms appears in \cite[Thm. 6]{MR1830474}.
First, we inherit all the orbits from the $\BC^2\ot \BC^2\ot \BC^{\ccc}$ case,
i.e., all the orbits from Table~\ref{table_orbits_22c}.
They become subvarieties of $Sub_{22\ccc}$ with the
same normal forms, pencils, ranks and border ranks
- the new dimensions are presented in Table~\ref{table_orbits_inherited_22c}.

In Table~\ref{table_orbits_inherited_22c}
  and below $\segduals \subset Sub_{ijk}$ denotes  the subvariety of $Sub_{ijk}$,
  obtained from the sub-fiber bundle of
  $\cS_{G(i,A)}\ot \cS_{G(j,B)}\ot \cS_{G(k,C)}$,
  whose fiber in  $A'\ot B'\ot C'$ (where  $\tdim A'=i$, $\tdim B'=j$, $\tdim C'=k$)
  is $\hat Seg(\BP A'^*\times \BP B'^*\times \BP C'^*)\dual\subset A'\ot B'\ot C'$.
In the special case $(i,j,k) = (\aaa,\bbb,\ccc)$, the variety $\segduals\subset Sub_{ijk}$ becomes $\xduals *$.

Table \ref{table_orbits_233} lists the  orbits in $\CC^2 \otimes \CC^3 \otimes \CC^\ccc$ that are contained in $Sub_{233}$,
  that is,  tensors in some $\CC^2 \otimes \CC^3 \otimes \CC^3$.

\begin{table}[htb]
$$
\begin{array}{|r|c|c|ll|c|c|}
\hline
\hline
\#&\text{orbit closure}&{\dim} &\text{Kronecker normal form} & \text{pencil} &   \ur &  \mathbf{R} \\
\hline
\hline
10& Sub_{133} &3\ccc & a_1\ot( b_1\ot c_1+    b_2\ot c_2+   b_3\ot c_3)
& \left(\begin{smallmatrix} s & & \\ & s & \\ & & s \end{smallmatrix} \right)
 &3&3\\
\hline
11& \segduals\subset Sub_{232} &2\ccc +6&
a_1\ot (b_1\ot c_1+ b_3\ot c_2) + a_2\ot b_2\ot c_1
& \left(\begin{smallmatrix} s & \\ t &   \\ &  s \end{smallmatrix} \right)
&3&3\\
\hline
12&  Sub_{232} &2\ccc +7& a_1\ot( b_1\ot c_1+    b_2\ot c_2 )+ a_2\ot(b_2\ot c_1+b_3\ot c_2)
& \left(\begin{smallmatrix} s & \\ t & s  \\ &  t \end{smallmatrix} \right)
 &3&3\\
\hline
13& & & a_1 \ot (b_1\ot c_1+b_2\ot c_3) + a_2\ot (b_1\ot c_2+b_3\ot c_3)
& \left(\begin{smallmatrix} s & t&\\   &  & s\\ &   & t\end{smallmatrix} \right)
& 3&4\\
\hline
14&  &  & a_1\ot( b_1\ot c_1+    b_2\ot c_2)+a_2\ot b_3\ot c_3
& \left(\begin{smallmatrix} s & &\\   & s & \\ &   & t\end{smallmatrix} \right)
&3 &3  \\
\hline
15&  &  & a_1\ot( b_1\ot c_1+    b_2\ot c_2+   b_3\ot c_3)+a_2\ot b_1\ot c_2
& \left(\begin{smallmatrix} s & t&\\   & s & \\ &   & s\end{smallmatrix} \right)
&3 & 4 \\
\hline
16&  &  &
\begin{matrix} a_1\ot( b_1\ot c_1+    b_2\ot c_2+   b_3\ot c_3)
\\
+a_2\ot( b_1\ot c_2+b_2\ot c_3)\end{matrix}
& \left(\begin{smallmatrix} s & t&\\   & s & t\\ &   & s\end{smallmatrix} \right)
&3 & 4 \\
\hline
17& \segduals\subset Sub_{233}  &3\ccc +7  &
           {a_1\ot( b_1\ot c_1+    b_2\ot c_2)+a_2\ot( b_1\ot c_2+b_3\ot c_3)} &
\left(\begin{smallmatrix} s & t&\\   & s & \\ &   & t\end{smallmatrix} \right)
 &3 &  4  \\
\hline
18& Sub_{233} &3\ccc +8&
           {a_1\ot (b_1\ot c_1+b_2\ot c_2)+ a_2\ot (b_2\ot c_2+  b_3\ot c_3)} &
\left(\begin{smallmatrix} s & &\\   & s+t & \\ &   & t\end{smallmatrix} \right)
  &3&3 \\
\hline
\hline
\end{array}
$$
\caption{\textbf{Orbits in   $\BC^2\ot \BC^3\ot \BC^{\ccc}$ contained in $Sub_{233}$}.
   Note that the  cases $11$ and $12$ are  are analogous to
orbits in  $\CC^2 \otimes \CC^2 \otimes \CC^{3}$.
   The unnamed orbits $13$--$16$ are various
   components of the singular locus of $\segduals\subset Sub_{233}$ (case $17$), see \cite{KWorbit} for descriptions.
} \label{table_orbits_233}
\end{table}

\begin{proposition}
  Table~\ref{table_orbits_233} lists   the orbits in $\PP(\CC^2 \ot \CC^3 \ot \CC^3)$,
  that are not contained in $Sub_{223}$.
\end{proposition}

\begin{proof}
  Let $[p] \in \PP(\CC^2 \ot \CC^3 \ot \CC^3)$, and let $p(A^*) = p((\CC^2)^*)$.
  If $\dim p(A^*)=1$, then $p$ must be as in $10$.
  Otherwise $\dim p(A^*)=2$ and the Kronecker normal form gives cases $11$--$13$,
    if there is at least one block of the form $L_{\epsilon}$ or $L_{\epsilon}^T$.
  Note that in the case $ \left(\begin{smallmatrix} s & \\ t &   \\ &  s+ \lambda t \end{smallmatrix} \right) $
    the eigenvalue may be set to zero  (as in case $7$) to  obtain case $11$.

 Now suppose we only have the block $s\Id_3 + tF$, for $F$ a $3\times 3$ matrix in its Jordan normal form.
  Then $F$ can be normalized to one of the six matrices in \eqref{3b3forms}.
  One of these matrices gives case $10$, while the remaining give cases $14$--$18$.

  Since $\sigma_3(Seg(\PP^1 \times \PP^2 \times \PP^2))$ fills out the ambient space (by
an easy application of Terracini's lemma or see \cite{MR2452824}),
    all the tensors listed in the table have border rank $3$.
  The ranks follow from Theorem \ref{jajathm}.
\end{proof}

We next consider   tensors contained in $Sub_{234}$ that are not
  contained in $Sub_{233}$.
These orbits are listed in Table~\ref{table_orbits_234}.

\begin{table}[htb]
$$
\begin{array}{|r|p{0.087\textwidth}|c|ll|c|c|}
\hline
\hline
\#&\text{orbit cl.}&{\dim} &\text{Kronecker normal form} & \text{pencil} &   \ur &  \mathbf{R} \\
\hline
\hline
19 & & &   a_1\ot (b_1\ot c_1+b_2\ot c_2+b_3\ot c_4)+ a_2\ot(b_1\ot c_2+b_2\ot c_3) &
\left(\begin{smallmatrix} s &t & & \\   & s & t & \\ & & &s\end{smallmatrix} \right)
 & 4 & 4
\\
\hline
20& & &   a_1\ot (b_1\ot c_1+b_2\ot c_3+b_3\ot c_4)+ a_2\ot b_1\ot c_2 &
\left(\begin{smallmatrix} s &t & & \\   &  & s & \\ & & &s\end{smallmatrix} \right)
 & 4 & 4
\\
\hline
21&  & &   a_1\ot (b_1\ot c_1+b_2\ot c_3+b_3\ot c_4)+ a_2\ot (b_1\ot c_2+ b_2\ot c_4) &
\left(\begin{smallmatrix} s &t & & \\   &  & s & t\\ & & &s\end{smallmatrix} \right)
& 4 & 5
\\
\hline
22& $\segduals\subset Sub_{234}$  & 4\ccc +6 &   a_1\ot (b_1\ot c_1+b_2\ot c_3)+ a_2\ot(b_1\ot c_2+b_3\ot c_4) &
\left(\begin{smallmatrix} s &t & & \\   &  & s & \\ & & &t\end{smallmatrix} \right)
& 4 & 4\\
\hline
23& $Sub_{234}$ & 4\ccc +7 &
\begin{matrix}a_1\ot (b_1\ot c_1+b_2\ot c_2+b_3\ot c_3)\\
+ a_2\ot(b_1\ot c_2+b_2\ot c_3+b_3\ot c_4)\end{matrix}
&\left(\begin{smallmatrix} s &t & & \\   & s & t & \\ & & s&t\end{smallmatrix} \right)
  &4&4 \\
\hline
\hline
\end{array}
$$
\caption{\textbf{Orbits in    $\BC^2\ot \BC^3\ot \BC^{\ccc}$ contained in $Sub_{234}$ but   not contained
 in $Sub_{233}$ or $Sub_{224}$}.
   The unlabeled orbit closures $19$--$21$
      are various components of the singular locus of $\segduals\subset Sub_{234}$, case $22$.}
   \label{table_orbits_234}
\end{table}

\begin{proposition}\label{prop_orbits_234}
  Table~\ref{table_orbits_234} lists  the orbits in $\PP(\CC^2 \ot \CC^3 \ot \CC^4)$
  that are not contained in $Sub_{233}$ or $Sub_{224}$.
\end{proposition}

\begin{proof}
  Let $[p] \in \PP(\CC^2 \ot \CC^3 \ot \CC^4)$, and let $p(A^*) = p((\CC^2)^*)$.
  If $\dim p(A^*)=1$, then $p$ must be  in $Sub_{233}$.
  Otherwise $\dim p(A^*)=2$ and by the Kronecker normal form,
    there must be at least one block of the form $L_{\epsilon}$
    (otherwise $p \in Sub_{233}$).
  Various configurations of the blocks give cases $19$-$23$.
  In all the cases the eigenvalues can be absorbed by a coordinate change,
    with perhaps the only non-trivial case
    $\left(\begin{smallmatrix} s &t & & \\   & s & t & \\ & & &s + \lambda t \end{smallmatrix} \right)$.
  In this case,  substitute $s' = s+ \lambda t$ to get
    $\left(\begin{smallmatrix} s'- \lambda t &t & & \\   & s'- \lambda t & t & \\ & & &s' \end{smallmatrix} \right)$.
  Then add $\lambda$ times the  third column to the second column to obtain
    $\left(\begin{smallmatrix} s'- \lambda t &t & & \\   & s' & t & \\ & & &s' \end{smallmatrix} \right)$.
  Add $\lambda$ times second column to the first column:
    $\left(\begin{smallmatrix} s'  &t & & \\  \lambda s' & s' & t & \\ & & &s' \end{smallmatrix} \right)$.
  Subtract  $\lambda$ times first row from the second row:
    $\left(\begin{smallmatrix} s'  &t & & \\   & s' - \lambda t & t & \\ & & &s' \end{smallmatrix} \right)$.
  Finally, add $\lambda$ times third column to the second column:
    $\left(\begin{smallmatrix} s'  &t & & \\   & s' & t & \\ & & &s' \end{smallmatrix} \right)$
    --- this is   case $19$.

  Since $\sigma_4(Seg(\PP^1 \times \PP^2 \times \PP^3))$ fills out the ambient space (by Terracini's lemma, or
see e.g. \cite{MR2452824}),
    all the tensors listed in the table have border rank $4$.
  The ranks follow from Theorem \ref{jajathm}.
\end{proof}

\begin{table}[htb]
$$
\begin{array}{|r|c|c|ll|c|c|}
\hline
\hline
\#&\text{orbit closure}&{\dim} &\text{Kronecker normal form}& \text{pencil} &   \ur &  \mathbf{R} \\
\hline
\hline
24& \xduals *  & 5\ccc +2 &
 a_1\ot (b_1\ot c_1+b_2\ot c_3+b_3\ot c_5) &
   \multirow{2}{*}{$\left(\begin{smallmatrix} s &t & & &\\   & & s & t &\\ & & & &  s\end{smallmatrix} \right)$}
 & 5& 5\\
&&&+a_2\ot(b_1\ot c_2+b_2\ot c_4)
&&&\\
\hline
25& Sub_{235}=\s_5(X) &5\ccc + 4&
 a_1\ot(b_1\ot c_1+b_2\ot c_2+b_3\ot c_4) &
   \multirow{2}{*}{$\left(\begin{smallmatrix} s &t & & &\\   & s & t & &\\ & & & s&t\end{smallmatrix} \right)$}
 &5&5 \\
&&&+a_2\ot(b_1\ot c_2+b_2\ot c_3+b_3\ot c_5)
  &&&\\
\hline
26 & \BP (A\ot B\ot C) &6\ccc -1&
 a_1\ot(b_1\ot c_1+b_2\ot c_3+b_3\ot c_5) &
\multirow{2}{*}{$\left(\begin{smallmatrix} s &t & & & &\\  & & s & t & &\\ & & & & s&t\end{smallmatrix} \right)$}
  &6&6 \\
&&&+a_2\ot(b_1\ot c_2+b_2\ot c_4+b_3\ot c_6)
& && \\
\hline
\hline
\end{array}
$$
\caption{\textbf{Orbits in $\BC^2\ot \BC^3\ot \BC^{\ccc}$  that are not contained in $Sub_{234}$}.
         When $\ccc=5$, $Sub_{235}=\BP (A\ot B\ot C)$ and case $26$ does not occur.}
   \label{table_orbits_23c}
\end{table}

Finally, we complete the list of orbits in $\BC^2\ot \BC^3\ot \BC^{\ccc}$.

\begin{proposition}
  Table~\ref{table_orbits_23c} lists the orbits in $\PP(\BC^2\ot \BC^3\ot \BC^{\ccc})$
    that are not contained in $Sub_{234}$.
\end{proposition}
\begin{proof}
  Since we need to fill a $3\times \ccc$ matrix with $\ccc \ge 5$, we need at least two blocks of the form $L_{\epsilon_i}$.
  Thus we either have two blocks $L_{1}$ and $F$ is a $1\times 1$ matrix (case $24$) or
       blocks $L_{1}$ and $L_2$ (case $25$), or three blocks $L_1$ (case $26$).
  The ranks follow from Theorem~\ref{jajathm}. The border rank is bounded from below by $5$ (cases $24$--$25$)
     and by $6$ in case $26$. It is also bounded from above by rank. This gives the result.
\end{proof}

\section{Application to $X=Seg(\pp 1\times v_2(\BP^{\bbb}))\subset \BP (A\ot S^2 \CC^{\bbb})$}\label{parsymsect}

This case is closely related to the determination of ranks of tensors with symmetric
matrix slices, an often-studied case in applications, see, e.g.,  \cite{MR2473189} and
the references therein.

P. Comon conjectured that the symmetric rank of a symmetric tensor is the same as its rank.
In \cite[\S4]{jabu_ginensky_landsberg_Eisenbuds_conjecture} Comon's conjecture
was generalized in several forms, in particular we asked if the analogous property
holds for   border rank
and for partially symmetric tensors. We show in the case at hand it does:

 Let $p\in \BC^{\aaa}\ot S^2\BC^{\bbb}$.
Write $\mathbf{R}_{\mathrm{ps}}(p)$ for the smallest $r$ such that $p$ is a sum of $r$ elements
of the form $a\ot b^2$, and $\ur_{\mathrm{ps}}(p)$ for the smallest $r$ such that it
is the limit of such.

\begin{theorem}\label{sympencilthm} Let $p\in \BC^2\ot S^2\BC^{\bbb}$.
Then $\mathbf{R}_{\mathrm{ps}}(p)=\mathbf{R} (p)$ and $\ur_{\mathrm{ps}}(p)=\ur(p)$.
  In particular, the maximum partially symmetric rank of an element of $\BC^2\ot S^2\BC^{\bbb}$
is $\lfloor \frac{3\bbb}2\rfloor$ (and the maximum partially symmetric border rank
had been known to be $\bbb$).
\end{theorem}

Note we always have $\mathbf{R}_{\mathrm{ps}}(p) \ge \mathbf{R} (p)$. 

The classification of pencils of quadrics is known, see  \cite[vol. 2, XII.6]{MR0107649}.
Every pencil of quadrics $p\in \BC^2\ot S^2\BC^{\bbb}$ is isomorphic to one built from blocks:
\begin{equation}\label{equ_Weierstass_normal_form}
  \begin{pmatrix}
      L^{sym}_{\epsilon_1} \\
                           &\ddots \\
                           &       & L^{sym}_{\epsilon_k} \\
                           &       &                      & G_{\lambda_1, \eta_1} \\
                           &       &                      &                      & \ddots \\
                           &       &                      &                      &        & G_{\lambda_l, \eta_l}\\
  \end{pmatrix},
\end{equation} 
with each block of the form either:
\[
  L^{sym}_{\epsilon}: = \begin{pmatrix}
                          0  &L_{\epsilon} \\
                          L^{T}_{\epsilon} &0 \\
                        \end{pmatrix}
\]
 where $L_{\epsilon}$ is as in \S\ref{sect_kronecker_normal_form}  (so that $L^{sym}_{\epsilon}$ is a $(2\epsilon+1)\times (2\epsilon +1)$-block) 
or
\[
  G_{\lambda, \eta} = \begin{pmatrix}
       0  & 0 & 0 &\dotsc & 0 &  t & s+ \lambda t \\
       0  & 0 & 0 &\dotsc & t & s+ \lambda t &  0 \\
       \vdots \\
       0  &  t & s+ \lambda t & \dotsc & 0 & 0 & 0\\
       t  &  s+ \lambda t & 0 & \dotsc & 0 & 0 & 0\\
       s+ \lambda t & 0   & 0 & \dotsc & 0 & 0 & 0
      \end{pmatrix}.
\]
The $\eta \times \eta$ blocks of the from $G_{\lambda, \eta}$ are analogous to the Jordan blocks 
  in the pencil $s\Id_f + t F$,
  but written in the other direction to maintain symmetry.
(The key point is that two pencils of
complex symmetric matrices are equivalent as symmetric pencils if and only if they are equivalent as
pencils of matrices, see \cite[vol. 2, XII.6, Thm. 6]{MR0107649}.)

We say that a square matrix $Q^i$ is a \emph{Hankel matrix}, if it is of the form
   $\begin{pmatrix}
      q^i_{1} & q^i_{2} & q^i_{3} & \dotsc \\
      q^i_{2} & q^i_{3} & \dotsc\\
      q^i_{3} & \vdots  & \ddots\\
      \vdots
   \end{pmatrix}$,
   that is an entry in $j$-th row and $k$-th column is equal to $q^i_{j+k}$.
   See \cite[p.~1]{partington_Hankel_operators} for more about Hankel matrices.

\begin{lemma}\label{lemma_no_common_divisors}
   Suppose $Q = 
   \begin{pmatrix}
      Q^1 \\
           & \ddots \\
                    && Q^l  
   \end{pmatrix}$
   is a symmetric $\bbb\times \bbb$ matrix with blocks $Q^i$, and
   whose entries are homogeneous polynomials of fixed degree $d$ in two variables $s,t$.
   Suppose there is no common divisor of all entries of $Q$ 
     and that each $Q^i$ is a Hankel matrix.
   Fix any finite set of complex numbers $\Lambda \subset \CC$.
   If $\mathbf{u} \in \CC^{\bbb}$ is a general vector, 
     then the polynomial $\mathbf{u}^T q \mathbf{u}$
     has distinct linear factors, none of which 
     equal to $s + \lambda t$ for any $\lambda \in \Lambda$.
\end{lemma}

\begin{proof}
   Due to the special form of the blocks $Q^i$, a general linear combination of all $q^i_{j}$
     can be obtained as $\mathbf{u}^T q \mathbf{u}$ for some $\mathbf{u}$.
   Since $q^i_{j}$ have no common divisor,
     the general linear combination has the required properties.
\end{proof}

The following Lemma is a symmetric version of Lemma~\ref{companlem}. 

\begin{lemma}\label{lem_rank_of_symmetric_pencil_with_distinct_eigenvalues}
   Let $p\in \BC^2\ot S^2\BC^{\bbb}$ be a pencil of quadrics consisting of the blocks:
   \[
     p = \begin{pmatrix}
             G_{\lambda_1, \eta_1} \\ 
                                 & \ddots  \\
                                 &         & G_{\lambda_l, \eta_l}   
           \end{pmatrix}
   \]
   with pairwise distinct $\lambda_i$ and $\eta_1 + \dotsb + \eta_l = \bbb$.
   \begin{itemize}
    \item  If all $\eta_i=1$ (i.e., $p$ is diagonal),
             then $\mathbf{R}_{\mathrm{ps}}(p) = \bbb$.
    \item If at least one $\eta_i \ge 2$, 
             then $\mathbf{R}_{\mathrm{ps}}(p) = \bbb +1$.
   \end{itemize}
\end{lemma}

\begin{proof}[Proof of Lemmas~\ref{companlem} and \ref{lem_rank_of_symmetric_pencil_with_distinct_eigenvalues}]
   In both Lemmas $p$ is the same pencil in the appropriate choice of coordinates.
   Thus consider $p$ as in Lemma~\ref{lem_rank_of_symmetric_pencil_with_distinct_eigenvalues}.
   If $p$ is diagonal, then clearly $\mathbf{R}_{\mathrm{ps}}(p) = \mathbf{R}(p) = \bbb$.
   Otherwise $\mathbf{R}_{\mathrm{ps}}(p) \ge  \mathbf{R}(p) \ge \bbb+1$.
   Thus it is sufficient to prove  $\mathbf{R}_{\mathrm{ps}}(p) \le \bbb+1$.

   Consider the classical  identity, where
      $\mathbf{B}$ is an $\bbb\times \bbb$ matrix, and $\mathbf{u},\mathbf{v}\in \BC^{\bbb}$:
   \[
     \det(\mathbf{B}+\mathbf{uv}^T) = (1 + \mathbf{v}^T\mathbf{B}^{-1}\mathbf{u})\,\det(\mathbf{B}).
   \]
   We will use this identity in the symmetric form
   \be\label{detlem}
      \det(\mathbf{B}+\mathbf{uu}^T) = \det(\mathbf{B}) + \mathbf{u}^Tcof\mathbf{(B)}  \mathbf{u}
   \ene
   where $cof\mathbf{(B)}$ denotes the cofactor matrix of $\mathbf{B}$.  We have
   \[
      cof(G_{\lambda_i, \eta_i}) = 
      -\begin{pmatrix}
       0  & 0 &\dotsc &  0 & (-s- \lambda_i t)^{\eta_i -1} \\
       0  & 0 &\dotsc & (-s - \lambda_i t)^{\eta_i -1} & (-s- \lambda_i t)^{\eta_i -2} t \\
       &&\vdots & \vdots & \vdots \\
       0  &  (-s -\lambda_i t)^{\eta_i -1} & \dotsc & (-s- \lambda_i t)^{2} t^{\eta_i -3}  & (-s- \lambda_i t) t^{\eta_i -2} \\
       (-s- \lambda_i t)^{\eta_i -1} &  (-s- \lambda_i t)^{\eta_i -2} t   & \dotsc & (-s- \lambda_i t) t^{\eta_i -2} & t^{\eta_i-1}
      \end{pmatrix},
   \]
   and $cof(p)$ has a blocks $\bigl(\prod_{j \ne i}-(-s - \lambda_j t)^{\eta_j}\bigr) cof(G_{\lambda_i, \eta_i})$
      centered around diagonal. For example, if $\bbb=5$, $\eta_1 =2$, and $\eta_2 =3$, then:
   \[
     p = 
      \begin{pmatrix}
         t              & \lambda_1 t + s  &    0            &    0            &  0              \\ 
        \lambda_1 t + s &  0               &    0            &    0            &  0              \\
         0              &  0               &    0            &    t            & \lambda_2 t + s \\
         0              &  0               &    t            & \lambda_2 t + s &  0              \\
         0              &  0               & \lambda_2 t + s &    0            &  0 
      \end{pmatrix}, \text{ and}
   \]
   \[
     cof(p) = 
      \left(\begin{smallmatrix}
         0              & (\lambda_1 t + s)(\lambda_2 t + s)^3  &    0            &    0            &  0              \\ 
        (\lambda_1 t + s)(\lambda_2 t + s)^3 &  -(\lambda_2 t + s)^3 t &    0            &    0            &  0              \\
         0              &  0               &    0            &    0           & (\lambda_1 t + s)^2(\lambda_2 t + s)^2 \\
         0              &  0               &    0            & (\lambda_1 t + s)^2(\lambda_2 t + s)^2 &  -(\lambda_1 t + s)^2(\lambda_2 t + s) t              \\
         0              &  0               & (\lambda_1 t + s)^2(\lambda_2 t + s)^2 &    -(\lambda_1 t + s)^2(\lambda_2 t + s) t            &  (\lambda_1 t + s)^2 t^2
      \end{smallmatrix}\right)
   \]
Since $\lambda_i$ are pairwise distinct, there is no common divisor of the entries of $cof(p)$ 
  viewed as homogeneous polynomials in $s$ and $t$.
Moreover $Q=cof(p)$ has the special form of Lemma~\ref{lemma_no_common_divisors}.
Thus $\mathbf{u}^Tcof(p)\mathbf{u}$ for general $\mathbf{u} \in \CC^{\bbb}$ 
   will have distinct factors, none of which are among the $(\lambda_j t + s)$.
Similarly, by just rescaling $\mathbf{u}$ if necessary, 
the polynomial
\[
    \det(p+t\mathbf{uu}^T) \stackrel{\text{by \eqref{detlem}}}{=} \det(p) + t \mathbf{u}^Tcof(p)  \mathbf{u} =
     (\lambda_1 t + s)\dotsm (\lambda_k t + s) + t \mathbf{u}^Tcof(p)  \mathbf{u}
\]
will be a polynomial with distinct linear factors.
Thus the perturbed pencil $p+t\mathbf{uu}^T$ is diagonalizable and
\[
  \mathbf{R}_{\mathrm{ps}}(p) \le  \mathbf{R}_{\mathrm{ps}}(p+t\mathbf{uu}^T) + \mathbf{R}_{\mathrm{ps}} (t\mathbf{uu}^T) \le \bbb+1.
\]
\end{proof}

Now we prove the analogue of Lemma~\ref{leplem}:
\begin{lemma}\label{lem_rank_of_L_sym}
   $\mathbf{R}_{\mathrm{ps}}(L^{sym}_{\epsilon}) = \mathbf{R}(L^{sym}_{\epsilon}) = 2\epsilon +2 $
\end{lemma}
\begin{proof}
   By Theorem~\ref{jajathm} we have $\mathbf{R}_{\mathrm{ps}}(L^{sym}_{\epsilon})   \ge \mathbf{R}(L^{sym}_{\epsilon}) =  2\epsilon +2 $, 
     thus it suffices to prove $\mathbf{R}_{\mathrm{ps}}(L^{sym}_{\epsilon}) \le 2\epsilon +2$.
   By Lemma~\ref{leplem} there exist $(\epsilon +1)$ rank $1$ 
     matrices $\fromto{M_0}{M_{\epsilon}}$
     such that $L_{\epsilon}(s,t) \subset \langle\fromto{M_0}{M_{\epsilon}}\rangle$ for all $s,t$.
   Write $N_i = \begin{pmatrix}
                  0 &  M_i \\ M_i^T & 0
                \end{pmatrix}$.
   Then $N_i$ is a symmetric matrix of rank $2$ and 
      $L_{\epsilon}^{sym}(s,t) \subset \langle\fromto{N_0}{N_{\epsilon}}\rangle$ for all $s,t$.
   Thus by Theorem~\ref{thm:can_calculate_rank_as_superspace_rank} 
      we have $\mathbf{R}_{\mathrm{ps}}(L^{sym}_{\epsilon})  \le \mathbf{R}(N_0) + \dotsb + \mathbf{R}(N_{\epsilon})  = 2(\epsilon +1)$.
\end{proof}

\begin{proof}[Proof of Theorem~\ref{sympencilthm}]
Write $p$ in the normal form of \eqref{equ_Weierstass_normal_form}.
By Theorem~\ref{jajathm}, 
\[
   \mathbf{R}_{\mathrm{ps}}(p) \ge \mathbf{R}(p) = \sum_{i=1}^k 2(\epsilon_i +1) + \sum_{j=1}^l \eta_j + \max_{\lambda \in \CC} \bigl(d(\lambda) \bigr),
\]
where $d(\lambda) = \# \set{ j \in \setfromto{1}{l} \mid \lambda_j = \lambda \text{ and } \eta_j \ge 2}$.
Let $G:= \begin{pmatrix}
             G_{\lambda_1, \eta_1} \\ 
                                 & \ddots  \\
                                 &         & G_{\lambda_l, \eta_l}   
   \end{pmatrix}$ and note:
\[
   \mathbf{R}_{\mathrm{ps}}(p) \le \sum_{i=1}^k \mathbf{R}_{\mathrm{ps}}(L^{sym}_{\epsilon_i}) 
                      + \mathbf{R}_{\mathrm{ps}}(G).
\]
By Lemma~\ref{lem_rank_of_L_sym} we have 
   $\mathbf{R}_{\mathrm{ps}}(L^{sym}_{\epsilon_i}) = 2(\epsilon_i + 1)$.
To estimate $\mathbf{R}_{\mathrm{ps}}(G)$,
   first chop off the diagonal blocks, i.e., pick $I:=\set{ i \in \setfromto{1}{l} \mid \eta_i  \ge 2}$
   and reorder blocks so that $G = \begin{pmatrix}
                     G_I &  \\  & D
                  \end{pmatrix}$
   where $G_I$ consists of the blocks $G_{\lambda_i, \eta_i}$ for $i \in I$, while $D$ consists of the remaining blocks
   (so $D$ is a diagonal pencil).
Now decompose $I= J_1 \sqcup  \dotsb \sqcup J_{\mu}$, where each $J_{\alpha}$ has the property: 
   if $i, j \in J_{\alpha}$, $i \ne j$, then $\lambda_i \ne \lambda_j$.
Thus each $G_{J_{\alpha}}$ has blocks with distinct eigenvalues,
   and by Lemma~\ref{lem_rank_of_symmetric_pencil_with_distinct_eigenvalues} we have  
   $\mathbf{R}_{\mathrm{ps}}(G_{J_{\alpha}}) = \sum_{j \in J_{\alpha}} \eta_j +1$
    and $\mathbf{R}_{\mathrm{ps}}(D)$ is the number of rows (or columns) of $D$.
Thus for each such decomposition   $I= J_1 \sqcup  \dotsb \sqcup J_{\mu}$:
\[
   \mathbf{R}_{\mathrm{ps}}(p) \le \sum_{i=1}^k  2(\epsilon_i + 1)
                      + \sum_{\alpha =1 }^m \mathbf{R}_{\mathrm{ps}}(G_{J_{\alpha}})  +  \mathbf{R}_{\mathrm{ps}}(D) =
   \sum_{i=1}^k  2(\epsilon_i + 1) + \sum_{j=1}^l \eta_j + \mu.
\]
It remains to pick $\mu$ as small as possible, which is $\mu = \max_{\lambda \in \CC} \bigl(d(\lambda) \bigr)$.

Explicitly, to obtain the upper bound $\mathbf{R}_{\mathrm{ps}}(p) = \lfloor \frac {3\bbb} 2\rfloor$,
   take a tensor consisting of $\lfloor \frac \bbb 2\rfloor$ blocks
   $G_{\lambda, 2 }$, all with the same eigenvalue $\l$, and if $\bbb$ is odd, add one $1 \times 1$ block $G_{\lambda', 1}$,
for any eigenvalue $\lambda'$. For instance, if $\lambda= \lambda'=0$, after reordering coordinates take:
\[
   \begin{pmatrix}
      t \Id_{\lfloor \frac \bbb 2\rfloor} & s \Id_{\lfloor \frac \bbb 2\rfloor}\\
      s \Id_{\lfloor \frac \bbb 2\rfloor} &  0 
   \end{pmatrix} 
    \text{ or }
   \begin{pmatrix}
      t \Id_{\lfloor \frac \bbb 2\rfloor} & s \Id_{\lfloor \frac \bbb 2\rfloor} & 0\\
      s \Id_{\lfloor \frac \bbb 2\rfloor} &  0                                  & 0\\
      0                                   &  0                                  & s
   \end{pmatrix}.
\]
\end{proof}

\section{Proofs of results in \S\ref{genlfactssect}}\label{agproofs}
The following  Lemma may be 
of interest in its own right, and the proof is a standard argument which is  \lq\lq well known to   experts.\rq\rq

\begin{lemma}\label{lem:hyperplane_section_is_nondegenerate}
   Let $Y \subset \PP W$ be a connected subvariety, which is not contained in any hyperplane in $\PP W$.
   Let $H \subset W$ be a hyperplane, which does not contain any irreducible component of $Y$ 
     (for example, $Y$ is irreducible).
   Then the scheme $Z:=Y \cap \PP H$ is not contained in any hyperplane in $\PP H$.
\end{lemma}

\begin{proof}
    Let $\ccI_{Y \subset \PP W}$ be the ideal sheaf of $Y$ in $\PP W$, and similarly for $\ccI_{Z \subset \PP H}$.
    The standard ring-ideal exact sequence
    $
       0 \to \ccI_{Y \subset \PP W} \to \ccO_{\PP W} \to \ccO_Y \to 0
    $
    leads to the cohomology long exact sequence:
    \[
       0 \to \underbrace{H^0(\ccI_{Y \subset \PP W})}_{= 0}
         \to \underbrace{H^0(\ccO_{\PP W})}_{\simeq \CC}
         \to \underbrace{H^0(\ccO_Y) }_{\simeq \CC, \text{ because Y is connected}}
         \to H^1(\ccI_{Y \subset \PP W})
         \to \underbrace{H^1(\ccO_{\PP W})}_{=0}.
    \]
    Thus  $h^1(\ccI_{Y \subset \PP W})=0$.

    Consider $h \in H^0(\ccO_{\PP W}(1))$ the defining equation of $H$. 
    By our assumptions, $h \tmod \ccI_{Y}$ is not a zero divisor, so we have a short exact sequence:
    \[
       0 \to  \ccI_{Y \subset \PP W} \stackrel{\cdot h}{\to} \ccI_{Y \subset \PP W}(1) \to \ccI_{Z \subset \PP H}(1) \to 0.
    \]
    Since $Y$ is not contained in any hyperplane, $h^0(\ccI_{Y \subset \PP W}(1)) =0$.
    Thus also  $h^0(\ccI_{Z \subset \PP H}(1)) = 0$,
       and $Z$ is not contained in any hyperplane in $\PP H$.
\end{proof}

\begin{lemma} \label{lem:bound_on_superspace_rank_for_disjoint}
Let $Y \subset \PP W$ be an irreducible subvariety not contained in a hyperplane,
   and let $E \subset W$ be a linear subspace disjoint from $Y$.
Then $\mathbf{R}_Y (E) \le \dim W - \dim Y$.
\end{lemma}
\begin{proof}
   Consider $H \subset W$, a general hyperplane containing $E$.
   By Bertini's Theorem $Z:=Y \cap H$ is reduced, see Part 3) of \cite[Thm. I.6.3]{jouanolou_Bertini}.
   By Part 1b) of the same Theorem, all components of $Z$ have dimension $\dim Y -1$.
   Furthermore, by Lemma~\ref{lem:hyperplane_section_is_nondegenerate} the points of $Z$ span $\PP H$.
   We will argue by induction on the dimension of $Y$.

   If $\dim Y =1$, then $Z$ is a finite collection of points spanning $\PP H$.
   We obtain $\mathbf{R}_Y(E) \le \mathbf{R}_Z(E) \le \dim H = \dim W - \dim Y$.

   If $\dim Y \ge 2$, then $Z$ is irreducible, see Part 4) of \cite[Thm. I.6.3]{jouanolou_Bertini}.
   By our induction hypothesis,  $\mathbf{R}_Y(E) \le \mathbf{R}_Z(E) \le \dim H - \dim Z =  \dim W - \dim Y$.
\end{proof}

\begin{proof}[Proof of Proposition~\ref{prop:bound_on_superspace_rank_LT}]
    Let $p \in A\otimes W$.
    By Theorem~\ref{thm:can_calculate_rank_as_superspace_rank} we have $\mathbf{R}_{X}(p) = \mathbf{R}_Y(p(A^*))$.
    Let $E' \subset p(A^*)$ be the linear space such that $\PP E'$ is the span of reduced points on $Y \cap p(A^*)$.
    Thus $\mathbf{R}_Y(E') = \dim E'$.
    If $E'=p(A^*)$, then $\mathbf{R}_Y(p(A^*)) = \dim E' \le \dim A$ and the claim holds, since $ \dim Y \le \dim W  -1$.
    Otherwise  $\dim E' \le \dim A -1$
       and choose a complement $E$ such that $p(A^*) = E' \oplus E$.
    Then $\mathbf{R}_Y(p(A^*)) \le \mathbf{R}_Y(E') + \mathbf{R}_Y(E) \le \dim A -1 + \mathbf{R}_Y(E)$
       and $\PP E$ is disjoint from $Y$.
    To conclude,   apply Lemma \ref{lem:bound_on_superspace_rank_for_disjoint}.
\end{proof}

\begin{proof}[Proof of Proposition~\ref{cor:dimensions_for_secants_of_product}]
  First assume $r = \dim A$.
  Then by Corollary \ref{cor:PGL_quotient}\ref{item:pi_is_the_PGL_quotient} and \ref{item:sigma_r_is_preimage},
  \[
    \dim \bigl(\sigma_r(Seg(\BP A \times X))\bigr) = \dim \bigl(\sigma_{r,r} (X)\bigr) + \dim PGL(A).
  \]
  Note that the inequality $\codim_{\BP W} X \ge r-1$
    is equivalent to $r \dim X \le \dim G (r, V)$.
  Thus by Theorem~\ref{thm_estimate_on_dim_of_Gr_secants},
    if $\codim_{\BP W} X \ge r-1$, then:
  \[
    \dim \bigl(\sigma_r(Seg(\BP A \times X))\bigr) = r \dim X + r^2 -1 = r(r+\dim X)-1
  \]
  as claimed.
  On the other hand, if $\codim_{\BP W} X < r-1$, then
  \[
    \sigma_r(Seg(\BP A \times X)) = \BP (A\otimes W) =  \sigma_r(Seg(\BP A \times \BP W)).
  \]

  Now assume $r < \dim A$.
  By Proposition~\ref{exprthm},
    the secant variety $\sigma_{r}(Seg(\BP A \times X))$ is swept out by smaller secant varieties:
  \begin{equation}\label{equ:secant_when_r_less_dim_A}
    \sigma_{r}(Seg(\BP A \times X)) =  \bigcup_{A' \subset A, \ \dim A' =r} \sigma_{r}(Seg(\BP A'  \times X)).
  \end{equation}
  Here the union is over the linear subspaces of $A$ of dimension $r$.
  We claim that if $p \in \sigma_{r}(Seg( \BP A \times X))$ is a general point,
    then there exists a unique $A'$ such that $p \in \sigma_{r}(Seg(\BP A' \times X))$.
  In fact, there is a unique $A'$ such that $p \in \BP (A' \otimes W)$.
  This is because $X$ is nondegenerate and $\dim W > r$,
    so $p = [a_1 \otimes y_1 + \dotsb + a_r \otimes y_r]$
    with the $y_i$ linearly independent in $  W$
    and the $a_i$ spanning $  A'$.
  Thus we can apply the above calculation in the case $r=\dim A$ with $A$ replaced with $A'$.
  If $\codim_{\BP W} X \ge r-1$, then:
  \begin{align*}
      \dim (\sigma_r(Seg(\BP A \times X))) & = \dim G(r, A) + \dim  (\sigma_r(Seg(\BP A' \times X)))\\
                                       & = r(\dim A -r) + r(r+\dim X)-1\\
                                       & = r (\dim A + \dim X) -1.
  \end{align*}
  If $\codim_{\BP W} X < r-1$, then we apply \eqref{equ:secant_when_r_less_dim_A}
    twice, including once with $X$ replaced by $\PP W$.
  \begin{align*}
   \sigma_{r}(Seg(\BP A \times X)) & =  \bigcup_{A' \subset A, \ \dim A' =r} \sigma_{r}(Seg(\BP A'  \times X)) \\
                                   & =  \bigcup_{A' \subset A, \ \dim A' =r} \BP(A' \otimes W)\\
                                   & =  \bigcup_{A' \subset A, \ \dim A' =r} \sigma_{r}(Seg(\BP A' \times \BP W))\\
                                   & =  \sigma_{r}(Seg(\BP A \times \BP W)).
  \end{align*}
\end{proof}

In particular, Proposition~\ref{cor:dimensions_for_secants_of_product} reproves \cite[Thm 2.4.2]{MR2392585}:
  if $X = Seg(\BP A_1\ctimes \BP A_n)$ with $\tdim A_s=\aaa_s$, $1\leq s\leq n$ and
  $\aaa_n >r\geq  \Pi_{i=1}^{n-1}\aaa_i- \sum_{i=1}^{n-1}\aaa_i-n+1$,
  then $\s_r(X)=\s_r(Seg(\BP(A_1\otimes\cdots\otimes   A_{n-1})\times \BP A_n))$.

\medskip

In \cite{BGI}, scheme-theoretic methods are used for studying rank.
For some varieties $X$, every point on $\s_r(X)$
is contained in the linear span of a  degree $r$ subscheme of $X$ (see \cite[Prop.~2.8]{BGI}).
For $X \subset \BP V$ consider the irreducible component $H_r$
of the Hilbert scheme $Hilb(X)$ containing schemes, which are $r$ distinct points with reduced structure.
Consider the rational map
\[
 \varphi : H_r \dashrightarrow G(r, V),
\]
which sends a subscheme $Z \subset X$ to its scheme-theoretic linear span.

\begin{proposition}
  If $\varphi$ is a regular map, i.e.,
   if each $Z \in H_r$ imposes independent conditions to linear forms,
  then for every $E \in \sigma_{r,k}(X)$ there exists a $0$-dimensional scheme $Z \subset X$ of degree $r$
  such that $E$ is contained in the scheme-theoretic span of $Z$.
\end{proposition}

Thus the methods used in \cite{BGI} may be used to study
 $Seg(\BP A \times v_d(\BP^n))$ and its secant varieties.

\bibliographystyle{amsplain}

\bibliography{BLtensorLek}

\end{document}

%% file: BLtensordefs.tex
\newtheoremstyle{custom}
  {3pt}
  {3pt}
  {\slshape}
  {}
  {\bfseries}
  {.}
  { }
   {}
\theoremstyle{custom}
\newtheorem{theorem}{Theorem}[section]

\newtheorem{proposition}[theorem]{Proposition}
\newtheorem{proposition/definition}[theorem]{Proposition/Definition}
\newtheorem{lemma}[theorem]{Lemma}
\newtheorem{corollary}[theorem]{Corollary}

\theoremstyle{definition}
\newtheorem{definition}[theorem]{Definition}

\newtheorem{example}[theorem]{Example}

\theoremstyle{remark}
\newtheorem{remark}[theorem]{Remark}

\newtheoremstyle{exercise}
  {3pt}
  {6pt}
  {}
  {}
  {\bfseries}
  {:}
  { }
   {}
\theoremstyle{exercise}
\newtheorem{exercise}[theorem]{Exercise}
\newtheoremstyle{exercises}
  {3pt}
  {6pt}
  {}
  {}
  {\bfseries}
  {:}
  {\newline}
   {}

\theoremstyle{exercise}
\newtheorem{exercises}[theorem]{Exercises}

\input epsf
\def\boxit#1{\vbox{\hrule height1pt\hbox{\vrule width1pt\kern3pt
  \vbox{\kern3pt#1\kern3pt}\kern3pt\vrule width1pt}\hrule height1pt}}

\def\BC{\mathbb C}\def\BO{\mathbb O}\def\BS{\mathbb S}

\def\BP{\mathbb P}
\def\pp#1{\mathbb P^{#1}}

\def\pp#1{{\mathbb P}^{#1}}
\def\tdim{{\rm dim}}

\def\hd{,...,}
\def\ww{\wedge}

\def\inv{{}^{-1}}

\def\cE{{\mathcal E}}

\def\cS{{\mathcal S}}

\def\CC{\mathbb C}

\def\11{\mathbf 1}
\def\PP{\mathbb P}

\def\l{\lambda}

\def\s{\sigma}

\def\ot{{\mathord{ \otimes } }}
\def\op{{\mathord{\,\oplus }\,}}
\def\otc{{\mathord{\otimes\cdots\otimes}\;}}

\def\ra{{\mathord{\;\rightarrow\;}}}

\def\dim{{\rm dim}\;}
\def\La#1{\Lambda^{#1}}

\def\op{\oplus}
\def\BO{\mathbb{O}}

\def\ep{\epsilon}
\def\op{\oplus}

\def\s{\sigma}

\def\l{\lambda}

\def\BP{\mathbb  P}
\def\BC{\mathbb  C}

\def\pp#1{\mathbb  P^{#1}}

\def\BS{\mathbb  S}

\def\ep{\epsilon}

\def\hd{, \hdots ,}

\def\inv{{}^{-1}}

\def\La#1{\Lambda^{#1}}

\def\pp#1{\mathbb  P^{#1}}

\def\ra{\rightarrow}

\def\tdim{\operatorname{dim}}

\def\tlim{\lim}
\def\tmod{\operatorname{mod}}

\def\tmin{\operatorname{min}}
\def\tmax{\operatorname{max}}

\def\ww{\wedge}

\def\ctimes{\times \cdots\times}

\def\be{\begin{equation}}
\def\ene{\end{equation}}
\def\aaa{{\mathbf{a}}}
\def\bbb{{\mathbf{b}}}
\def\ccc{{\mathbf{c}}}\def\bv{{\mathbf{v}}}\def\bw{{\mathbf{w}}}

\def\dual{{^\vee}}
\def\xduals#1{X^{\vee}_{#1}}

\newcommand{\Id}{\operatorname{Id}}







%% file: BLtensorLek.bbl
\def\cdprime{$''$} \def\cprime{$'$} \def\cprime{$'$} \def\cprime{$'$}
  \def\Dbar{\leavevmode\lower.6ex\hbox to 0pt{\hskip-.23ex \accent"16\hss}D}
  \def\cprime{$'$} \def\cprime{$'$} \def\cdprime{$''$}
  \def\Dbar{\leavevmode\lower.6ex\hbox to 0pt{\hskip-.23ex \accent"16\hss}D}
  \def\cprime{$'$} \def\cprime{$'$} \def\cprime{$'$} \def\cprime{$'$}
  \def\Dbar{\leavevmode\lower.6ex\hbox to 0pt{\hskip-.23ex \accent"16\hss}D}
  \def\cprime{$'$} \def\cprime{$'$}
\providecommand{\bysame}{\leavevmode\hbox to3em{\hrulefill}\thinspace}
\providecommand{\MR}{\relax\ifhmode\unskip\space\fi MR }
\providecommand{\MRhref}[2]{%
  \href{http://www.ams.org/mathscinet-getitem?mr=#1}{#2}
}
\providecommand{\href}[2]{#2}
\begin{thebibliography}{10}

\bibitem{MR2452824}
Hirotachi Abo, Giorgio Ottaviani, and Chris Peterson, \emph{Induction for
  secant varieties of {S}egre varieties}, Trans. Amer. Math. Soc. \textbf{361}
  (2009), no.~2, 767--792. \MR{2452824 (2010a:14088)}

\bibitem{BGI}
Alessandra Bernardi, Alessandro Gimigliano, and Monica Id{\`a}, \emph{Computing
  symmetric rank for symmetric tensors}, J. Symbolic Comput. \textbf{46}
  (2011), no.~1, 34--53. \MR{2736357}

\bibitem{jabu_ginensky_landsberg_Eisenbuds_conjecture}
J.~Buczy\'nski, A~Ginensky, and J.~M. Landsberg, \emph{Determinantal equations
  for secant varieties and the {E}isenbud-{K}oh-{S}tillman conjecture},
  arXiv:1007.0192, 2010.

\bibitem{BCS}
Peter B{\"u}rgisser, Michael Clausen, and M.~Amin Shokrollahi, \emph{Algebraic
  complexity theory}, Grundlehren der Mathematischen Wissenschaften
  [Fundamental Principles of Mathematical Sciences], vol. 315, Springer-Verlag,
  Berlin, 1997, With the collaboration of Thomas Lickteig. \MR{99c:68002}

\bibitem{carlini_chipalkatti_Waring_for_several_forms}
Enrico Carlini and Jaydeep Chipalkatti, \emph{On {W}aring's problem for several
  algebraic forms}, Comment. Math. Helv. \textbf{78} (2003), no.~3, 494--517.
  \MR{1998391 (2005b:14097)}

\bibitem{MR2392585}
M.~V. Catalisano, A.~V. Geramita, and A.~Gimigliano, \emph{On the ideals of
  secant varieties to certain rational varieties}, J. Algebra \textbf{319}
  (2008), no.~5, 1913--1931. \MR{2392585}

\bibitem{chiantini_coppens_Grassmannians_of_secant_varieties}
Luca Chiantini and Marc Coppens, \emph{Grassmannians of secant varieties},
  Forum Math. \textbf{13} (2001), no.~5, 615--628. \MR{1858491 (2002g:14079)}

\bibitem{chipalkatti_geramita_Artin_level_algebras}
J.~V. Chipalkatti and A.~V. Geramita, \emph{On parameter spaces for {A}rtin
  level algebras}, Michigan Math. J. \textbf{51} (2003), no.~1, 187--207.
  \MR{1960928 (2004c:13025)}

\bibitem{chipalkatti_Waring_locus}
Jaydeep~V. Chipalkatti, \emph{The {W}aring locus of binary forms}, Comm.
  Algebra \textbf{32} (2004), no.~4, 1425--1444. \MR{2100365 (2006e:11046)}

\bibitem{ciliberto_cools_Grassmann_secant_extremal_varieties}
Ciro Ciliberto and Filip Cools, \emph{On {G}rassmann secant extremal
  varieties}, Adv. Geom. \textbf{8} (2008), no.~3, 377--386. \MR{2427466
  (2009e:14087)}

\bibitem{MR2517853}
P.~Comon, J.~M.~F. ten Berge, L.~De~Lathauwer, and J.~Castaing, \emph{Generic
  and typical ranks of multi-way arrays}, Linear Algebra Appl. \textbf{430}
  (2009), no.~11-12, 2997--3007. \MR{2517853}

\bibitem{MR2447451}
Pierre Comon, Gene Golub, Lek-Heng Lim, and Bernard Mourrain, \emph{Symmetric
  tensors and symmetric tensor rank}, SIAM J. Matrix Anal. Appl. \textbf{30}
  (2008), no.~3, 1254--1279. \MR{2447451 (2009i:15039)}

\bibitem{cools_Singular_locus_of_Grassmann_secant}
Filip Cools, \emph{On the singular locus of {G}rassmann secant varieties},
  Bull. Belg. Math. Soc. Simon Stevin \textbf{16} (2009), no.~5, Linear systems
  and subschemes, 799--803. \MR{2574361 (2011a:14108)}

\bibitem{MR2447444}
Vin de~Silva and Lek-Heng Lim, \emph{Tensor rank and the ill-posedness of the
  best low-rank approximation problem}, SIAM J. Matrix Anal. Appl. \textbf{30}
  (2008), no.~3, 1084--1127. \MR{2447444 (2009h:15013)}

\bibitem{MR2009896}
Carla Dionisi and Claudio Fontanari, \emph{Grassman defectivity \`a la
  {T}erracini}, Matematiche (Catania) \textbf{56} (2001), no.~2, 245--255
  (2003), PRAGMATIC, 2001 (Catania). \MR{2009896 (2004h:14059)}

\bibitem{Friedland3tensor}
Shmuel Friedland, \emph{On the generic rank of 3-tensors}, preprint (2008).

\bibitem{MR0107649}
F.~R. Gantmacher, \emph{The theory of matrices. {V}ols. 1, 2}, Translated by K.
  A. Hirsch, Chelsea Publishing Co., New York, 1959. \MR{0107649 (21 \#6372c)}

\bibitem{GKZ}
I.~M. Gel{\cprime}fand, M.~M. Kapranov, and A.~V. Zelevinsky,
  \emph{Discriminants, resultants, and multidimensional determinants},
  Mathematics: Theory \& Applications, Birkh\"auser Boston Inc., Boston, MA,
  1994. \MR{95e:14045}

\bibitem{MR0453768}
D.~Ju. Grigor{\cprime}ev, \emph{The algebraic complexity of computing a pair of
  bilinear forms}, Zap. Nau\v cn. Sem. Leningrad. Otdel. Mat. Inst. Steklov.
  (LOMI) \textbf{47} (1974), 159--163, 188, 193, Investigations on linear
  operators and the theory of functions, V. \MR{0453768 (56 \#12022)}

\bibitem{MR519843}
D.~Yu. Grigoriev, \emph{Multiplicative complexity of a pair of bilinear forms
  and of the polynomial multiplication}, Mathematical foundations of computer
  science, 1978 ({P}roc. {S}eventh {S}ympos., {Z}akopane, 1978), Lecture Notes
  in Comput. Sci., vol.~64, Springer, Berlin, 1978, pp.~250--256. \MR{519843
  (80d:68052)}

\bibitem{MR521052}
Joseph Ja'Ja', \emph{Optimal evaluation of pairs of bilinear forms}, Conference
  {R}ecord of the {T}enth {A}nnual {ACM} {S}ymposium on {T}heory of {C}omputing
  ({S}an {D}iego, {C}alif., 1978), ACM, New York, 1978, pp.~173--183.
  \MR{521052 (80j:68032)}

\bibitem{jouanolou_Bertini}
Jean-Pierre Jouanolou, \emph{Th\'eor\`emes de {B}ertini et applications},
  Progress in Mathematics, vol.~42, Birkh\"auser Boston Inc., Boston, MA, 1983.
  \MR{725671 (86b:13007)}

\bibitem{Kacorb}
V.~G. Kac, \emph{Some remarks on nilpotent orbits}, J. Algebra \textbf{64}
  (1980), no.~1, 190--213. \MR{81i:17005}

\bibitem{KWorbit}
Witold Kra{\'s}kiewicz and Jerzy Weyman, \emph{Modules with a finite number of
  orbits}, preprint (2009).

\bibitem{MR0398167}
Jean-Claude Lafon, \emph{Optimum computation of {$p$} bilinear forms}, Linear
  Algebra and Appl. \textbf{10} (1975), 225--240. \MR{0398167 (53 \#2022)}

\bibitem{MR2383305}
J.~M. Landsberg, \emph{Geometry and the complexity of matrix multiplication},
  Bull. Amer. Math. Soc. (N.S.) \textbf{45} (2008), no.~2, 247--284.
  \MR{2383305 (2009b:68055)}

\bibitem{LMseries}
J.~M. Landsberg and L.~Manivel, \emph{Series of {L}ie groups}, Michigan Math.
  J. \textbf{52} (2004), no.~2, 453--479. \MR{2069810 (2005f:17009)}

\bibitem{LMsel}
J.~M. Landsberg and Laurent Manivel, \emph{Construction and classification of
  complex simple {L}ie algebras via projective geometry}, Selecta Math. (N.S.)
  \textbf{8} (2002), no.~1, 137--159. \MR{1890196 (2002m:17006)}

\bibitem{LMsecb}
\bysame, \emph{Generalizations of {S}trassen's equations for secant varieties
  of {S}egre varieties}, Comm. Algebra \textbf{36} (2008), no.~2, 405--422.
  \MR{2387532}

\bibitem{LTrank}
J.~M. Landsberg and Zach Teitler, \emph{On the ranks and border ranks of
  symmetric tensors}, Found. Comput. Math. \textbf{10} (2010), no.~3, 339--366.
  \MR{2628829 (2011d:14095)}

\bibitem{MR1304906}
D.~Mumford, J.~Fogarty, and F.~Kirwan, \emph{Geometric invariant theory}, third
  ed., Ergebnisse der Mathematik und ihrer Grenzgebiete (2) [Results in
  Mathematics and Related Areas (2)], vol.~34, Springer-Verlag, Berlin, 1994.
  \MR{1304906 (95m:14012)}

\bibitem{MR2205865}
Lior Pachter and Bernd Sturmfels (eds.), \emph{Algebraic statistics for
  computational biology}, Cambridge University Press, New York, 2005.
  \MR{2205865 (2006i:92002)}

\bibitem{MR1830474}
P.~G. Parfenov, \emph{Orbits and their closures in the spaces
  {${\mathbb{C}}^{k_1}\otimes\dots\otimes{\mathbb{C}}^{k_r}$}}, Mat. Sb.
  \textbf{192} (2001), no.~1, 89--112. \MR{1830474 (2002b:14057)}

\bibitem{Strassen505}
V.~Strassen, \emph{Rank and optimal computation of generic tensors}, Linear
  Algebra Appl. \textbf{52/53} (1983), 645--685. \MR{85b:15039}

\bibitem{MR0521168}
Volker Strassen, \emph{Vermeidung von {D}ivisionen}, J. Reine Angew. Math.
  \textbf{264} (1973), 184--202. \MR{0521168 (58 \#25128)}

\bibitem{Teichert}
L.~Teichert, \emph{Die komplexit\"at von bilinearformpaaren \"uber beliebigen
  korpern}, PhD thesis, unpublished (1986).

\bibitem{MR2473189}
Jos M.~F. Ten~Berge, Alwin Stegeman, and Mohammed Bennani~Dosse, \emph{The
  {C}arroll and {C}hang conjecture of equal {I}ndscal components when
  {C}andecomp/{P}arafac gives perfect fit}, Linear Algebra Appl. \textbf{430}
  (2009), no.~2-3, 818--829. \MR{2473189}

\bibitem{Terracini2}
A.~Terracini, \emph{Sulla rappresentazione delle coppie di forme ternaire
  mediante somme di potenze di forme lineari}, Ann. Mat. Pur. ed. appl
  \textbf{XXIV, III} (1915), 91--100.

\end{thebibliography}
